\documentclass[12pt, leqno]{amsart}
\usepackage{amsthm,amsmath}
\usepackage{amssymb,xy}
\xyoption{arrow}

\xyoption{matrix}

\theoremstyle{plain}
\newtheorem{theorem}{Theorem}[section]

\newtheorem{lemma}[theorem]{Lemma}

\newtheorem{cor}[theorem]{Corollary}
\newtheorem{prop}[theorem]{Proposition}

\theoremstyle{remark}

\theoremstyle{definition}
\newtheorem{claim}[theorem]{Claim}
\newtheorem{defn}[theorem]{Definition}
\newtheorem{Remark}[theorem]{Remark}

\numberwithin{equation}{section}
\numberwithin{theorem}{section}

\newcommand{\NN}{\mathbb N}

\newcommand{\ZZ}{\mathbb Z}
\newcommand{\QQ}{\mathbb Q}

\newcommand{\A}{\mathcal A}

\newcommand\hfld[2]{\smash{\mathop{\hbox to 10mm{\rightarrowfill}}
     \limits^{\scriptstyle#1}_{\scriptstyle#2}}}
\newcommand\hflg[2]{\smash{\mathop{\hbox to 10mm{\leftarrowfill}}
     \limits^{\scriptstyle#1}_{\scriptstyle#2}}}

\def\geo{{\partial_\infty}} 
\def\De{\Delta}
\def\Del{\Delta}
\def\al{\alpha}
\def\be{\beta}
\def\ga{\gamma}
\def\R{{\mathbb R}}
\def\la{\lambda}
\def\La{\Lambda}

\def\RA{\Rightarrow}
\def\B{\mathcal B}
\def\T{\mathcal T}
\def\F{\mathcal F}
\def\I{\mathcal I}

\def\ol{\overline}
\def\ul{\underline}

\def\<{\langle}
\def\>{\rangle}
\def\D{\partial}
\def\ov{\overrightarrow}
\def\a{\mathbf a}
\def\tits{\D_{Tits}}
\def\La{\Lambda}
\def\si{\sigma}

\title{Ideal triangles in Euclidean buildings \\
and branching to Levi subgroups}
\author{Thomas J. Haines, Michael Kapovich, John J. Millson}

\date{}

\begin{document}

\date{\today}

\begin{abstract}
Let $\ul{G}$ denote a connected reductive group, defined and split over $\ZZ$, and let $\ul{M} \subset \ul{G}$ denote a Levi subgroup.  
In this paper we study varieties of geodesic triangles with fixed vector-valued side-lengths $\alpha,\beta,\gamma$ in the Bruhat-Tits buildings associated to $\ul{G}$, along with varieties of ideal triangles associated to  the pair $\ul{M}\subset \ul{G}$.  The ideal triangles have a fixed side containing a fixed base vertex and a fixed infinite vertex $\xi$ such that other infinite side containing $\xi$ has fixed ``ideal length'' $\lambda$ and the remaining finite side has fixed length $\mu$.  
We establish an isomorphism between varieties in the second family and certain varieties in the first 
family (the pair $(\mu,\la)$ and the triple 
$(\alpha,\beta,\gamma)$ satisfy a certain relation).  We apply these results to the study of the Hecke ring of $\ul{G}$ and the  
restriction homomorphism ${\mathcal R}(\ul{\widehat{G}})\to {\mathcal R}(\ul{\widehat{M}})$ between representation rings.  We deduce some new saturation theorems for constant term coefficients and for the structure constants of the restriction homomorphism.
\end{abstract}

\maketitle

\markboth{T. Haines, M. Kapovich, J. Millson}
{Ideal triangles and branching to Levi subgroups}

\section{Introduction}
Let $\ul{G}$ be a connected reductive group, defined and split over $\ZZ$, and fix a split maximal torus $\ul{T}$ also defined over $\ZZ$.  Let $\widehat{G} = \widehat{G}(\mathbb C)$ denote the Langlands dual group of $\ul{G}$, and let $\mathcal R(\widehat{G})$ denote its representation ring.  Let ${\mathcal H}_{G}$ denote the (spherical) Hecke ring associated to $\ul{G}(\mathbb F_q(\!(t)\!))$, as described in section \ref{general}.  
The goal of this paper is to understand various connections between the rings ${\mathcal H}_{{G}}$ 
and ${\mathcal R}(\widehat{{G}})$.  Both come with bases and associated structure constants $m_{\al,\be}(\ga), \, n_{\al,\be}(\ga)$ parameterized by the same set, namely triples $\alpha, \beta, \gamma$ of  $\ul{G}$-dominant elements of the cocharacter lattice of $\ul{T}$. 
Moreover, given any Levi subgroup $\ul{M} \subset \ul{G}$, 
we have the constant term homomorphism
$$
c^G_M: {\mathcal H}_G \to {\mathcal H}_M
$$
and the restriction homomorphism
$$
r^G_M: {\mathcal R}(\widehat{G}) \to {\mathcal R}(\widehat{M});
$$
cf.~section \ref{general}.  Assuming $\ul{M}$ contains $\ul{T}$, both maps can be described by collections of constants $c_\mu(\la)$ and $r_\mu(\la)$, where $\mu$ resp. $\la$ ranges over the $\ul{G}$-dominant resp. $\ul{M}$-dominant cocharacters of $\ul{T}$; cf.~loc.~cit.  In this paper we are studying connections between entries appearing in the following table.  

\begin{table}[h]
\caption{Constants associated to ${\mathcal H}$ and ${\mathcal R}$}

\begin{tabular}{|c|c|}
\hline 
$c_\mu(\la)$  & $r_\mu(\la)$      		\\
\hline  
$m_{\al,\be}(\ga)$ & $n_{\al,\be}(\ga)$ 	\\ 
\hline
\end{tabular} 
\label{table}
\end{table}

The connection between the entries in the bottom row was studied in \cite{KLM3} and \cite{KM2}. In this paper we will establish connections between the entries in the top row, the entries in the first column and the entries in the second column. As a corollary we will establish saturation results for the entries in the top row. 

It was established in \cite{KLM3} that $m_{\al,\be}(\ga)$ ``counts'' the number of $\mathbb F_q$-rational points in the variety of triangles ${\mathcal T}(\al,\be;\ga)$ in the Bruhat-Tits building of $\ul{G}(\bar{\mathbb F}_p(\!(t)\!))$. Similarly, fixing a parabolic subgroup $\ul{P} = \ul{M} \cdot \ul{N}$ with Levi factor $\ul{M}$, we will see that 
$c_\mu(\la)$ counts (up to a certain factor depending only on $\ul{P}$, $q$, and $\la$) 
the number of $\mathbb F_q$-points in the variety of ideal triangles $\I\T(\la, \mu; \xi)$ with the 
ideal vertex $\xi$ fixed by $\ul{P}(\bar{\mathbb F}_p(\!(t)\!))$ (see section \ref{general} for the definition). Given $\la, \mu$, we will find a certain range of $\al, \be, \ga$ depending on $\la,\mu$, so that 
that the varieties  $\I\T(\la, \mu; \xi)$ and ${\mathcal T}(\al, \be; \ga)$ are naturally isomorphic over $\mathbb F_p$, thereby providing a geometric explanation for the numerical equalities 
\begin{equation} \label{num1}
c_\mu(\lambda) q^{\<\rho_N,\la\>} |K_{M,q}\cdot x_\la|= m_{\al, \be}(\ga)
\end{equation}
and 
\begin{equation} \label{num2}
r_{\mu}(\la) = n_{\alpha,\beta}(\gamma).
\end{equation} 

Let us state our main results a little more precisely.  
The equality (\ref{num2}) has a short proof using the Littelmann path models for each side (see section \ref{Littelmann_sec}), and this proof gave rise to the definition of the inequality $\nu \geq^P \mu$ (see section \ref{precise} for the definition).  Now fix any coweight $\nu$ that satisfies this inequality, so that in particular $\nu+ \lambda$ will be $\ul{G}$-dominant for any $\ul{M}$-dominant cocharacter $\lambda$ appearing as a weight in $V^{\widehat{G}}_\mu$.  We can now state our first main theorem (Theorem \ref{key}).

\begin{theorem} \label{thm_key0}  Suppose $\mu,\la$ are as above $\nu$ is any auxiliary cocharacter satisfying $\nu \geq^P \mu$.  Then there is an isomorphism of $\mathbb F_p$-varieties
$$
\T(\nu + \la, \mu^*;\nu) \cong \I\T(\la, \mu;\xi).
$$
\end{theorem}

As detailed in section \ref{bounds_sec}, the number of top-dimensional irreducible components of $\T(\nu+\la,\mu^*;\nu)$ (resp. $\I\T(\la,\mu;\xi)$) is simply the multiplicity $n_{\nu+\la,\mu^*}(\nu)$ (resp. $r_{\mu}(\la)$).  Similarly, in section \ref{geom_int_sec} we show that the number of $\mathbb F_q$-points on $\T(\nu+\la,\mu^*;\nu)$ (resp.~$\I\T(\la,\mu;\xi)$) is given by $m_{\nu+\la,\mu^*}(\nu)$ (resp. $c_{\mu}(\la)$, up to a factor).  Thus, Theorem \ref{thm_key0} implies the numerical equalities (\ref{num1}) and (\ref{num2}).  This is stated more completely in Theorem \ref{main}.

Because of the homogeneity properties of the inequality $\nu \geq^P \mu$, these equalities mean that saturation theorems for $n_{\al,\be}(\ga)$ resp.~$m_{\al,\be}(\ga)$ imply saturation theorems for $r_\mu(\la)$ resp.~$c_{\mu}(\la)$.  The following summarizes part of Corollary \ref{corollary:main}.

\begin{cor} \label{corollary:main0}
The quantities $c_{\mu}(\la)$ satisfy a saturation theorem with saturation factor $k_\Phi$, and the 
quantities $r_\mu(\la)$ satisfy a saturation theorem with saturation factor $k^2_\Phi$.
\end{cor}

For the precise formulation of these results we refer the reader to section \ref{precise}. 
Since there are two groups of mathematicians interested in the results of this paper, 
we will present both 
algebraic and geometric interpretations of the concepts and results.

\medskip Here are a few words on the relation of this paper to the prior work. 
In the earlier works \cite{KLM1}, \cite{KLM2}, \cite{KLM3}, Leeb and the second and third named authors studied geometric and representation-theoretic problems {\bf Q1}, {\bf Q2}, {\bf Q3}, {\bf Q4} (see page 1 of \cite{KLM3} for the precise formulations). 
In the present paper we study the analogues of {\bf Q3}, {\bf Q4} for group pairs $(G, M)$. 
The problems analogous to {\bf Q1}, {\bf Q2} for group pairs $(G, M)$ were studied in \cite{BeSj} and \cite{F} respectively. The paper \cite{BeSj} actually studies the problem for general group pairs $G,M$, where $G$ is a reductive group and $M$ is any reductive subgroup. 

\medskip
Let us give an outline of the contents of this article.  In section \ref{general} we recall some standard definitions and notation and we also define the notions of based triangles and based ideal triangles in the building.  In section \ref{precise} we state our main results.  In section \ref{Littelmann_sec} we give a simple proof of one of main results using Littelmann paths, and thereby explain the origin of the inequality $\nu \geq^P \mu$.  In section \ref{geo} we give a detailed study of based ideal triangles and the corresponding Busemann functions.  We translate Theorem \ref{key} into a statement about retractions and study those retractions in sections \ref{retractions_sec} and \ref{sharp_sec}; the proof of Theorem \ref{key} is given in section \ref{key_thm_proof_sec}.  The rest of the paper until section \ref{equidim_sec} is directed toward the proof of Theorem \ref{main}.  We prove some a priori bounds on dimensions of the varieties of (ideal) triangles in section \ref{bounds_sec}; these give geometric interpretations for the numbers $n_{\al,\be}(\ga)$ and $r_\mu(\la)$ appearing in Theorem \ref{main}.  Section \ref{geom_int_sec} likewise gives necessary geometric interpretations for the quantities $m_{\al,\be}(\ga)$ and $c_\mu(\la)$.  In section \ref{Thm_main_proof} we put the pieces together and prove Theorem \ref{main} and Corollary \ref{corollary:main}.  In section \ref{equidim_sec} we provide some equidimensionality statements which are related to those given in \cite{Haines2} for fibers of convolution morphisms. Finally, in the Appendix (section \ref{appendix}) we give an alternative, more geometric, proof of the main ingredient in the proof of Theorem \ref{thm_key0}, namely, the equality of the retractions $\rho_{-\nu,\De_G-\nu}$ and $\rho_{K_P,\De_M} =b_{\xi,\De_M}$ on each geodesic 
$\ol{oz}$ of $\De_G$-length $\mu$, when $\nu\ge^P\mu$.

\medskip
{\bf Acknowledgments.} All three authors were supported  by the NSF FRG grant, DMS-05-54254 and DMS-05-54349.  The first author was also supported by NSF grant DMS-09-01723, the second author by NSF grant DMS-09-05802, and the  third author by NSF grant DMS-09-07446.   In addition the first author was supported by a University of Maryland Graduate Research Board Semester Award and the third author by the Simons 
Foundation. The authors are  grateful to A.~Berenstein and S.~Kumar  
for explaining how to relate $n_{\al,\be}(\ga)$ and $r_\mu(\la)$, for some appropriate choices of $\alpha, \beta, \gamma,\mu,\lambda$, using representation theory. 

\section{Notation and definitions}\label{general}

\subsection{Algebra}

In what follows, all the algebraic groups will be over $\ZZ$. Let $\ul{G}$ be a split connected reductive group, and let $\ul{T}\subset \ul{G}$ be a split maximal torus.   Fix a Levi subgroup $\ul{M} \subset \ul{G}$ which contains $\ul{T}$.

Choose a parabolic subgroup $\ul{P} \subset \ul{G}$ which has $\ul{M}$ as a Levi factor.  Let $\ul{P}= \ul{M} \cdot \ul{N}$ be a Levi splitting. Then choose a Borel subgroup $\ul{B}$ of $\ul{G}$ which contains $\ul{T}$ and is contained in $\ul{P}$.  
Let $\ul{U}\subset \ul{B}$ be the unipotent radical of $\ul{B}$.  We then have $\ul{N}\subset \ul{U}$. 

Let $\Phi$ denote the set of roots for $(\ul{G}, \ul{T})$, let $\Phi_N$ denote the set of roots for $\ul{T}$ 
appearing in ${\rm Lie}(\ul{N})$ and let $\Phi_M$ denote all roots in $\Phi$ which belong to $\ul{M}$.  
We let $Q(\Phi^\vee)$ denote the coroot lattice and and $P(\Phi^\vee)$ the coweight lattice. 

The choice of $\ul{B}$ (resp. $\ul{B}_M := \ul{B} \cap \ul{M}$) gives a notion of positive (co)root, and $\ul{G}$-dominant (resp. $\ul{M}$-dominant) element of 
$\mathcal A := X_*(\ul{T}) \otimes \mathbb R$.  Let $\rho$ denote\footnote{We also use the symbol $\rho$ in the context of retractions of buildings (see section \ref{retractions_sec}) but no confusion should result from this.} the half-sum of the $\ul{B}$-positive roots $\Phi^+$. Similarly, we define 
$\rho_N$ resp. $\rho_M$ to be the half-sums of all roots in $\Phi_N$ resp.~positive roots in $\Phi_M$.  Recall that $W$, the Weyl group of $\ul{G}$, acts by reflections on 
$\mathcal A$ with fundamental domain $\Del_G$ which is the convex hull of the $\ul{G}$-dominant coweights. Also, we define 
$\Del_M$ as the convex hull of the $\ul{M}$-dominant coweights, so that $\Del_G\subset \Del_M$ and $\De_M$ is the fundamental domain of $W_M$, the Weyl group 
$N_{\ul{M}}(\ul{T})/\ul{T}$ for $\ul{M}$. We let $\widetilde{W}$ denote the extended affine Weyl group of $\ul{G}$, i.e., 
$\widetilde{W}=\Lambda\rtimes W$, where $\Lambda := X_*(\ul{T})$. 

Given $\la\in  X^*(\ul{T})$ or $X_*(\ul{T})$, define $\la^* := -w_0\la$, where $w_0\in W$ is the longest element. Note that $\rho^*=\rho$.   Set $k_\Phi={\rm lcm}(a_1,...,a_l)$, where $\sum_{i=1}^l a_i \al_i=\theta$ is the highest root and $\al_i$ are the simple roots of $\Phi$.   Let $\langle \cdot, \cdot \rangle: X^*(\ul{T}) \times X_*(\ul{T}) \rightarrow \mathbb Z$ denote the canonical pairing.

We define $\widehat{G}:=\ul{\widehat{G}}(\mathbb C)$ and, similarly, 
define $\widehat{M}$ and $\widehat{T}$.   Having fixed the inclusions $\ul{G} \supset \ul{M} \supset \ul{T}$, we can arrange that we also have $\widehat{G} \supset \widehat{M} \supset \widehat{T}$.  We will identify $X^*(\widehat{T})$ with $X_*(T)$ and roots of $(\widehat{G},\widehat{T})$ with coroots of $(G,T)$. 

Let $V^{\widehat{G}}_\mu$ denote the irreducible representation of $\widehat{G}$ having highest weight $\mu$. Let 
$\Omega(\mu)$ denote the set of $\widehat{T}$-weights in $V^{\widehat{G}}_\mu$, i.e., the intersection of the convex hull of $W\cdot \mu$ with the character lattice of $\widehat{T}$.    We shall also think of $\Omega(\mu)$ as consisting of certain cocharacters of $\ul{T}$.

For $\mu, \lambda, \alpha, \beta, \gamma \in X^*(\widehat{T})$, define
\begin{align}
r_\mu(\lambda) &= {\rm dim}~ {\rm Hom}_{\widehat{M}}(V^{\widehat{M}}_\lambda,\, V^{\widehat{G}}_\mu) \\
n_{\alpha, \beta}(\ga) &= {\rm dim} ~ {\rm Hom}_{\widehat{G}}(V^{\widehat{G}}_\alpha \otimes 
V^{\widehat G}_\beta, \, V^{\widehat G}_\ga).
\end{align}
Let $\mathcal R(\widehat{G})$ denote the representation ring of $\widehat{G}$.  The numbers $n_{\alpha, \beta}(\gamma)$ are the structure constants for $\mathcal R(\widehat{G})$, relative to the basis of highest weight representations $\{ V^{\widehat{G}}_\alpha \}$.  Similarly, the $r_\mu(\lambda)$ are the structure constants for the restriction homomorphism $\mathcal R(\widehat{G}) \rightarrow \mathcal R(\widehat{M})$.

\medskip

Let $\mathbb F_q$ denote the finite field with $q = p^n$ elements (for a prime $p$), let $k$ denote the algebraic closure 
$\bar{\mathbb F}_p=\bar{\mathbb F}_q$. Define the local function fields $L = k(\!(t)\!)$ and $L_q = \mathbb F_q(\!(t)\!)$ and their rings of integers $\mathcal O = k[\![t]\!]$ and $\mathcal O_q = \mathbb F_q[\![t]\!]$.

Let $G:= \ul{G}(L)$ and $G_q := \ul{G}(L_q)$, and similarly, we define $B, M, N, P, T, U$ and $B_q, M_q$, etc.  (Note that in what follows, we will often abuse notation and write $G, M, B$, etc., instead of $G_q, M_q, B_q$, 
etc. (resp. $\ul{G}, \ul{M}, \ul{B}$, etc.), letting context dictate what is meant.)  

Set $K := \ul{G}(\mathcal O)$ and $K_q := \ul{G}(\mathcal O_q)$.  These are maximal bounded subgroups of $G = \ul{G}(L)$ resp. $G_q := \ul{G}(L_q)$.  Set $K_M:=K\cap M$, $K_{M,q} = K_M \cap M_q$, and $K_P := N \cdot K_M$.

Let ${\mathcal H}_G = C_c(K_q \backslash G_q/K_q)$ and ${\mathcal H}_M = C_c(K_{M,q} \backslash M_q/K_{M,q})$
denote the spherical Hecke algebras of $G_q$ and $M_q$ respectively (they depend on $q$, but we will suppress this in our notation $\mathcal H_G$).  Convolution is defined using the Haar measures giving $K_q$ respectively
$K_{M,q}$ volume 1.  For the parabolic subgroup $P = MN$ of $G$, the constant term
homomorphism $c^G_M: {\mathcal H}_G \rightarrow {\mathcal H}_M$ is defined by the formula
$$ 
c^G_M(f)(m) = \delta_P(m)^{-1/2} \int_{N_q} f(nm) dn,
$$
for $m \in M_q$.  
Here, the Haar
measure on $N_q$ is such that $N_q \cap K$ has volume 1.  Further, letting $|\cdot|$ denote the normalized absolute value on $L_q$, we have $\delta_P(m) := |{\rm det}({\rm Ad}(m); {\rm Lie}(N))|$.  We define in a similar
way $\delta_B$,
$\delta_{B_M}$, $c^G_T$, and $c^M_T$.  If $U_M := U \cap M$, then we have $U = U_M \, N$, and so
$$
\delta_B(t) = \delta_P(t) \delta_{B_M}(t)
$$ for $t \in T_q$, and
$$
c^G_T(f)(t) = (c^M_T \circ c^G_M)(f)(t).
$$

The map $c^G_T$ (resp. $c^M_T$) is the Satake isomorphism $S^G$ for $G$ (resp.
$S^M$ for $M$).  Thus, the following diagram commutes:
\begin{equation} \label{ct_diag}
\xymatrix{ {\mathcal R}(\widehat{G})    \ar[r]^{\cong} \ar[d]_{rest.} & \mathbb C[X_*(T)]^W
\ar[d]_{incl.} & {\mathcal H}_G \ar[l]_{\,\,\,\,\,\,\,\,\,\, S^G}
\ar[d]_{c^G_M} \\ {\mathcal R}(\widehat{M}) \ar[r]^{\cong} & \mathbb C[X_*(T)]^{W_M} & {\mathcal H}_M
\ar[l]_{\,\,\,\,\,\,\,\,\,\, S^M}. }
\end{equation}

Given a cocharacter $\la \in X_*(\ul{T})$, we set $t^\la:=\la(t)$, where $t\in L$ is the variable.  
For a $G$-dominant coweight $\mu$, let $f^G_\mu = {\rm char}(K_q t^\mu K_q)$, the characteristic function of
the coset $K_q t^\mu K_q$.  Let $f^M_\lambda$ have the analogous meaning. When convenient, we will omit the symbols $G$ and $M$ in the notation for $f^G_\mu, f^M_\mu$. For $G$-dominant coweights $\alpha, \beta, \gamma$ define the {\em structure
constants} for the algebra ${\mathcal H}_{G}$ by
$$
f_\alpha * f_\beta =\sum_{\gamma} m_{\alpha, \beta}(\gamma) f_\gamma.
$$
Note that the $m_{\alpha,\beta}(\gamma)$ are functions of the parameter $q$, however we will suppress this dependence.

For $G$-dominant $\mu$ and $M$-dominant $\lambda$, we define $c_\mu(\la)$ by 
$$
c^G_M(f_\mu^G)=\sum_{\la} c_\mu(\la) f_\la^M
$$
Like the $m_{\alpha, \beta}(\gamma)$, the numbers $c_\mu(\la)$ depend on $q$, but we will suppress this.

\subsection{Definition of based (ideal) triangles in buildings}

Let $\B=\B_G$ denote the Bruhat-Tits building of $G$. This is a Euclidean building.  It is not locally finite, because $L$ has infinite residue field; however this will cause us no problems.
This building has a distinguished special point $o$ fixed by $K$.  We consider it as the ``origin'' in the base apartment $\A$ corresponding to $\ul{T}$.  Later on, we shall need to consider also the base alcove ${\bf a}$ in $\A$: it is the unique alcove of $\A$ whose closure contains $o$ and which is contained in the dominant Weyl chamber $\De_G$.

In what follows, we will sometimes write $\De$ in place of $\De_G$.  Recall that the $\De$-distance $d_\De(x,y)$ in $\B$ is defined as follows. Given $x, y\in \B$, find an apartment $\A'\subset  \B$ containing $x, y$. 
Identify $\A'$ with the model apartment $\A$ using an isomorphism $\A' \to \A$.   
Then project the vector $\ov{xy}$ in $\A$ to a vector $\ov{\la}$ the positive chamber $\De_G\subset \A$, so that $x$ corresponds to the origin $o$, the tip of $\De_G$. Then 
$$
d_\De(x,y):=\la.$$
Thus, $d_\De(x,y)=d_\De(y,x)^*$. Given a coweight $\la\in \De \cap \Lambda$ and $t^\la\in T$, we let $x_\la:= t^\la \cdot o$, a point in $\B$. Then $d_\De(o,x_\la)= W\cdot \la\cap \De$. 
For $x\in \B$ and $\la\in \De$ we define the $\la$-sphere $S_\la(x)=\{y\in \B: d_\De(x,y)=\la\}$. In the case 
when $x=o$ and $\la\in \De \cap \La$, we have
$$
S_\la(o)=K\cdot x_\la. 
$$

\begin{defn}
Given $\al, \be, \ga\in \De \cap \Lambda$ define the space of {\em based ``disoriented'' triangles} ${\mathcal T}(\al, \be; \ga)$ to be the set of triangles $[o,y;x_\ga]$ 
with vertices $o, y, x_\ga$, so that 
$$
d_\De(o, y)=\al, \, d_\De(y, x_\ga)=\be.
$$
Note that only the point $y$ is varying.
\end{defn}
Observe that ${\mathcal T}(\al,\be;\ga)$ can be identified with the subset of the usual set of oriented triangles $\mathcal T(\al,\be,\ga^*)$ whose final edge is $\ov{x_\ga o}$.  Also, it is easy to see that
$$
\mathcal T(\alpha,\beta;\gamma) = Kx_\alpha \cap t^\gamma Kx_{\beta^*}
$$
under the identification given by the map $[o,y;x_\ga]\to y$. 

\medskip

We need to define a variant of the distance function $d_{\De}(-,-)$, where one of the points is ``at infinity'' in a particular sense we will presently describe.  We let $\ol{xy}$ denote the unique geodesic segment in $\B$ connecting 
$x$ to $y$. We will always assume that such segments (and all geodesic rays in $\B$) are parameterized by arc-length. We let $\D_{Tits}\B$ denote the Tits boundary of $\B$, which is a spherical 
building. The points of $\tits{\B}$ could be defined as equivalence classes of geodesic rays in $\B$: 
two rays are equivalent if they are {\em asymptotic}, i.e., are within bounded distance from each other. A ray in $\B$ is denoted $\ol{x\xi}$ where $x$ is its initial point and $\xi\in \tits \B$ represents the corresponding point in $\tits\B$.  
 
One says that two rays $\ga_1(t), \ga_2(t)$ in $\B$ are {\em strongly asymptotic} if $\ga_1(t)=\ga_2(t)$ for all sufficiently large $t$. 

Each parabolic subgroup $P$ of $G$ fixes a certain cell in $\D_{Tits}\B$. In what follows, we 
will pick a generic point $\xi$ in that cell. Then $P$ is the stabilizer of $\xi$ in $G$. 


Now assume that $M$ is a Levi factor of the parabolic $P$ corresponding to $\xi$.  
By analogy with the definition of Busemann functions in metric geometry, we will define vector-valued Busemann functions 
(normalized at $o$) 
$$
b_{\xi,\De_M}: \B_G \to \De_M. 
$$
We refer the reader to Section \ref{geo} for the precise definition. Intuitively, $b_{\xi,\De_M}(y)$ measures the  
$\De_M$-distance from $\xi$ to $y$ relative to the $\De_M$-distance from $\xi$ to $o$. A fundamental property (to be proved in Lemma \ref{Busemann_comp_lem}) is that
\begin{equation} \label{Busemann_alg_char}
b_{\xi,\De_M}(y)=\la \iff y\in K_P x_\la.
\end{equation}
This gives an algebraic characterization of the function $b_{\xi,\De_M}$, and also shows that it agrees with the retraction $\rho_{K_P,\De_M}$ which we define and study in subsection \ref{rho_K_P_Delta_M}. 

We can now define the space of based ideal triangles. 

\begin{defn}
Fix coweights $\la\in \De_M, \mu\in \De_G$ and a generic point $\xi$ in the face of $\D_{Tits} \B$ fixed by $P$. Then we define the set of {\em based ideal triangles} $\I\T(\la,\mu; \xi)$ to consist 
of the triples $o, y, \xi$, where 
$$
d_{\De}(o,y)=\mu, \,\, b_{\xi,\De_M}(y)=\la.
$$
Note that once again, only $y$ is varying.
\end{defn}
In other words, in view of (\ref{Busemann_alg_char}), we have the purely algebraic characterization (proved in Corollary \ref{IT_comp_cor})
$$
\I\T(\la,\mu; \xi)= S_{\mu}(o) \cap K_P \cdot x_\la. 
 $$

\subsection{Affine Grassmannians and algebraic structure of (ideal) triangle spaces}

We need to endow ${\mathcal T}(\alpha,\beta;\ga)$ and $\mathcal{IT}(\lambda,\mu;\xi)$ with the structure of algebraic varieties defined over $\mathbb F_p$.  To do so we will realize them as subsets of the affine Grassmannian.

The {\em affine Grassmannian} ${\rm Gr}^G := G/K$ will be considered as the $k$-points of an ind-scheme defined over 
$\mathbb F_p$.  We can identify this with the orbit $G\cdot o\subset \B_G$. (If $\ul{G}$ is semisimple then 
$G\cdot o$ is contained in the vertex set of $\B_G$, in general it is a subset of the skeleton of the smallest dimension in 
the polysimplicial complex  $\B_G$.) 

For any $G$-dominant cocharacter $\mu$, let $x_\mu = t^\mu K/K$, a point in ${\rm Gr}^G$.  It is well-known that the closure $\overline{Kx_\mu}$ of the $K$-orbit $Kx_\mu = S_\mu(o)$ in the affine Grassmannian is the union
$$
\overline{S_\mu(o)} = \coprod_{\mu_0 \preceq \mu} S_{\mu_0}(o).
$$
Here $\mu_0$ ranges over $G$-dominant cocharacters in $X_*(T)$, and the relation $\mu_0 \preceq \mu$ means, by definition, that $\mu - \mu_0$ is a sum of positive coroots.  

Each $\overline{S_\mu(o)}$ (resp. $S_\mu(o)$) is a projective (resp. quasi-projective) variety of dimension $\langle 2\rho, \mu \rangle$, defined over $\mathbb F_p$.  Therefore ${\rm Gr}^G$, the union of the projective varieties $\overline{S_\mu(o)}$, is an ind-scheme defined over $\mathbb F_p$.

Now, $K$ is the set of $k$-points in a group scheme defined over $\mathbb F_p$ (namely, 
the positive loop group $L^{\geq 0}(\ul{G})$) which acts (on the left) on ${\rm Gr}^G$ in an obvious way.  
The orbits $Kx_\mu$ are automatically locally-closed in the (Zariski) topology on ${\rm Gr}^G$, and are defined over 
$\mathbb F_p$.

Moreover, the group $K_P = N K_M$ we defined earlier is the $k$-points of an ind-group-scheme defined over $\mathbb F_p$ which also acts on ${\rm Gr}^G$.  The orbit spaces $K_P x_\lambda$ are neither finite-dimensional nor finite-codimensional in general, however, since they are orbits under an ind-group, they are still automatically locally closed in ${\rm Gr}^G$.

By the above discussion, our spaces of triangles can be viewed as intersections of orbits inside ${\rm Gr}^G$
\begin{align*}
\mathcal T(\al,\be;\ga) &= Kx_\alpha \cap t^\ga Kx_{\beta} \\
{\mathcal IT}(\lambda,\mu;\xi) &= K_Px_\lambda \cap Kx_\mu
\end{align*}
and as such each inherits the structure of a finite-dimensional, locally-closed subvariety defined over $\mathbb F_p$.  Thus, it makes sense to count $\mathbb F_q$-points on these varieties.

\begin{Remark}
The Bruhat-Tits building $\B_{G_q}$ for the group $G_q$ isometrically embeds in $\B_G$ as a sub-building.  It is the fixed-point set for the natural action of the Galois group ${\rm Gal}(k/\mathbb F_q)$ on $\B_{G}$. The orbit $G_q \cdot o \subset \B_{G_q}$ can be identified with $G_q/K_q$ and thus with the set of $\mathbb F_q$-points in ${\rm Gr}^{G}$. Accordingly, the sets of 
$\mathbb F_q$-points in $T(\al,\be;\ga)$ and  ${\mathcal IT}(\lambda,\mu;\xi)$ then become spaces of based triangles and based ideal triangles in $\B_{G_q}$. Then ``counting'' the numbers of triangles in $\B_{G_q}$ computes structure constants for ${\mathcal H}_G$ and (up to a factor) the constant term map $c^G_M$. On the other hand, algebro-geometric considerations are more suitable for the varieties of triangles in ${\rm Gr}^{G}\subset \B_G$, since the field $k$ is algebraically closed. Therefore, in this paper (unlike \cite{KLM3}), we almost exclusively work with the building $\B_G$ rather than $\B_{G_q}$. 
\end{Remark}

\section{Statements of results}  \label{precise} 

We fix cocharacters $\mu \in \De_G$ and $\lambda \in \De_M$.  In order to state our results we need the following definition.  Recall that $\langle \cdot, \cdot \rangle: X^*(\ul{T}) \times X_*(\ul{T}) \rightarrow \mathbb Z$ denotes the canonical pairing.

\begin{defn}  Suppose $\mu, \nu \in X_*(\ul{T})$.  We write $\nu \geq^P \mu$ if
\begin{itemize}
\item $\langle \alpha, \nu \rangle = 0$ for all roots $\alpha$ appearing in ${\rm Lie}(\ul{M})$;
\item $\langle \alpha, \nu + \lambda \rangle \geq 0$ for all $\lambda \in \Omega(\mu)$ and $\alpha \in \Phi_N$.  
\end{itemize}
\end{defn}

Note that this relation satisfies a semigroup property: if $\nu_1 \geq^P \mu_1$ and $\nu_2 \geq^P \mu_2$, then $\nu_1 + \nu_2 \geq^P \mu_1 + \mu_2$.  It is also homogeneous: for every integer $z\geq 1$, we have $\nu \geq^P \mu \Longleftrightarrow z\nu \geq^P z\mu$. 
\medskip

\begin{theorem}\label{key}
  Let $\mu,\lambda$ be as above.  Then for any cocharacter $\nu$ with $\nu \geq^P \mu$, we have an equality of subvarieties in ${\rm Gr}^G$
$$\T(\nu+\lambda, \mu^*;\nu) = t^\nu\Big(\I\T(\lambda, \mu;\xi) \Big).$$
In particular, the varieties $\T(\nu+\lambda, \mu^*;\nu)$ and $\I\T(\lambda, \mu;\xi)$ are naturally isomorphic as $\mathbb F_p$-varieties.
\end{theorem}

 For the next results, recall that $k_\Phi= {\rm lcm}(a_1,...,a_l)$, where $\sum_{i=1}^l a_i \al_i=\theta$ is the highest root and $\al_i$ are the simple roots of $\Phi$.

\begin{theorem}\label{main}
For $\la\in \De_M,\, \mu\in \De_G$ as above and for any $\nu$ with $\nu \geq^P \mu$, set $\al:=\nu+\la,\, \be:=\mu^*, \, \ga:=\nu$. Then: 

(i) (First column of Table \ref{table}) 
$$
c_\mu(\lambda) q^{\<\rho_N,\la\>} |K_{M,q}\cdot x_\la|= m_{\al,\be}(\ga). 
$$ 

(ii) (Second column) $r_\mu(\lambda) = n_{\al,\be}(\ga)=n_{\nu, \mu}(\nu + \lambda)$. 

(iii) (First row) $$r_\mu(\la)\ne 0 \RA c_\mu(\la)\ne 0 \RA r_{k_\Phi \mu}(k_\Phi \la)\ne 0.$$ 
\end{theorem}

Assume now that $\mu-\la$ (or, equivalently, $\la + \mu^*$) belongs to the coroot lattice 
of $\ul{G}$.     
  
\begin{cor}\label{corollary:main}
i. (Semigroup property for $r$)
The set of $(\mu,\lambda)$ for which $r_\mu(\la)\neq 0$ is a semigroup.

\medskip

ii. (Uniform Saturation for $c$) 
$$c_{N\mu}(N\la)\ne 0 \hbox{~~for some~~} N\ne 0 \quad \RA c_{k_{\Phi}\mu}(k_\Phi \la)\ne 0.$$  


\medskip 
iii. (Uniform Saturation for $r$) 
$$r_{N\mu}(N\la)\ne 0 \hbox{~~for some~~} N\ne 0 \quad \RA r_{k_\Phi^2\mu}(k_\Phi^2 \la)\ne 0.$$  
In particular, for $\ul{G}$ of type A, the set of $(\mu, \lambda)$  such that $r_{\mu}(\lambda) \neq 0$, is saturated.
\end{cor}

\begin{Remark}
One can improve (using results of \cite{KM1}, \cite{KKM}, \cite{BK}, \cite{S} on saturation for the structure constants for the 
representation rings ${\mathcal R}(\widehat{G})$) the constants $k_\Phi$ as follows: 

One can replace $k_\Phi$ (in ii.) and $k_\Phi^2$ (in iii.) by:

(a) $k=2$ for $\Phi$ of type $B, C, G_2$, 

(b) $k=1$ for $\Phi$ of type $D_4$,

(c)  $k=2$ for $\Phi$ of type $D_n, n\ge 6$. 

Conjecturally (see \cite{KM1}), one can use $k=1$ for all simply-laced root systems and $k=2$ for all non-simply-laced. 
\end{Remark}

\begin{Remark}  The paper \cite{Haines2} states saturation results and conjectures for the numbers $m_{\alpha, \beta}(\gamma)$ and $n_{\alpha, \beta}(\gamma)$, when $\alpha$ and $\beta$ are sums of $\ul{G}$-dominant minuscule cocharacters.  Suppose that $\mu$ is a sum of $\ul{G}$-dominant minuscule cocharacters.   Then we conjecture that the implications in ii. and iii. above hold, where $k_\Phi$ and $k_{\Phi}^2$ are replaced by $1$.
\end{Remark}

\section{Relation to Littelmann's path model} \label{Littelmann_sec}

There is a very short proof of Theorem \ref{main} (ii) using Littelmann's path models, and our discovery of this proof was one of the starting points for this project.  It gave rise to the notion of the inequality $\nu \geq^P \mu$ which plays a key role for us.  For this reason, we present this proof here.  A nice reference for the Littelmann path models used here is \cite{Li}.

We will prove the result in the following form: If $\nu \geq^P \mu$, then 
$r_\mu(\lambda) = n_{\nu, \mu}(\nu + \la)$. 

Fix $\mu \in \De_G \cap \Lambda$.  Write $\mathbb B_\mu$ for the set of all {\rm type $\mu$ LS-paths}, that is, the set of all paths in $\mathcal A$ which result by applying a finite sequence of ``raising'' resp. ``lowering'' operators $e_i$ resp. $f_i$ to the path $\ov{o\mu}$ (the straight-line path from the origin to $\mu$).  Note that each path in $\mathbb B_\mu$ starts at the origin $o$ and lies inside ${\rm Conv}(W_\mu)$, the convex hull of the Weyl group orbit of $\mu$.

For any $x \in \mathcal A$, we can consider $x + \mathbb B_\mu$, the set of all type $\mu$ paths originating at $x$.
Let $\mathbb B_\mu(x,y)$ denote the set of type $\mu$ paths which originate at $x$ and terminate at $y$.

Now consider $\lambda \in \De_M \cap \Lambda$.  The Littelmann path model for $r_\mu(\lambda)$ states that 
\begin{equation} \label{Lit_rest}
r_\mu(\lambda) = |\mathbb B_\mu(o,\lambda) \cap \De_M|,
\end{equation}
the cardinality of the set of type $\mu$ paths originating at $o$, terminating at $\lambda$, and contained entirely in $\De_M$.

Now consider any $\nu \in \De_G \cap \Lambda$ such that $\nu + \la \in \De_G$.  The Littelmann path model for $n_{\nu, \mu}(\nu + \lambda)$ states that 
\begin{equation} \label{Lit_tens}
n_{\nu, \mu}(\nu + \la) = |\mathbb B_\mu(\nu, \nu+\lambda) \cap \De_G|
\end{equation}
the cardinality of the set of type $\mu$ paths originating at $\nu$, and terminating at $\nu + \la$, and contained entirely inside $\De_G$.

Now assume $\nu \geq^P \mu$.  This is defined precisely so that we have
$$
\nu + (\mathbb B_{\mu}(o,\la) \cap \De_M) = \mathbb B_\mu(\nu, \nu+ \la) \cap \De_G.
$$
The equality of the quantities in (\ref{Lit_rest}) and (\ref{Lit_tens}) is now obvious. See figure below. \qed

\begin{figure}[h]\label{LS-paths}
\begin{displaymath}
\begin{xy}
(-40,0)*{}="A"; (0, 10)*{}="B"; (40, 0)*{}="C"; (20, 20)*{}="D"; (-20, 20)*{}="E"; (-60,0)*{}="AA"; (60, 0)*{}="CC";
"A"*{\bullet}; "B"; "C"*{\bullet}; "D"*{\bullet};  
{\ar "A";"C"};
{\ar "C"; "B"};
{\ar "B"; "D"};
{\ar "C"; "D"};
{\ar "A"; "D"};
{\ar @{-}  "A"; "E"}; {\ar @{-}  "AA"; "CC"};
(-44, 2)*{  x_{-\nu}};  (42, 2)*{  o}; (20, 23)*{  x_\la};
(32, 12)* { \la}; (0,-4)* {  \nu}; (-4,16)* {  \la+\nu}; (12,12)* {  p_\mu}; (-14, 22)* {  \De_G - \nu} ; (57, 3)*{  \De_M}
\end{xy}
\end{displaymath}
\caption{The broken path $p_\mu$ from $o$ to $x_\la$ is an LS path of type $\mu$. Then the Littelmann bigon with the sides $p_\mu$ and 
$\ol{o x_\la}$ yields a Littelmann triangle with the geodesic sides $\ol{x_{-\nu} o}$, $\ol{x_{-\nu} x_\la}$ and the broken side $p_\mu$.} 
\end{figure}
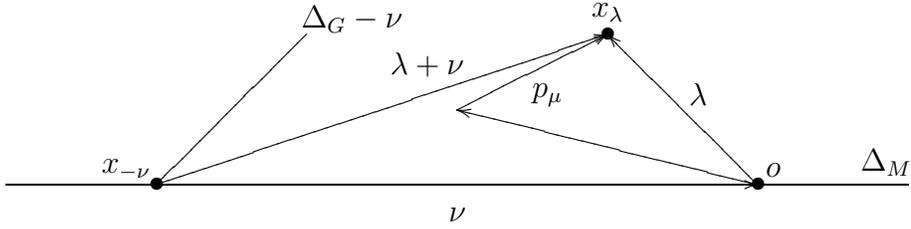

\section{Ideal triangles and vector-valued Busemann functions}\label{geo}
\subsection{Based ideal triangles with ideal side-length $\lambda$ and side-length $\mu$}

As promised earlier in this paper we now make precise the definition of vector-valued Busemann function.  
Recall that we fixed a parabolic subgroup $P\subset G$, with a Levi subgroup $M$ and a point $\xi\in \tits \B$ fixed by $P$, where $\B=\B_G$. Then we also have an embedding $\B_M\subset \B_G$, where $\B_M$ splits as the product 
$$
\B_M=F\times {\mathcal Y}_M
$$
where $F$ is a flat in $\B$ and ${\mathcal Y}_M$ is a Euclidean building containing no flat factors. Pick a geodesic $\gamma\subset F$ asymptotic to 
$\xi$, and oriented so that the subray $\ga(\R_+)$ is asymptotic to $\xi$. 
Then ${\mathcal B}_M= {\mathcal P}(\gamma)$, the subbuilding which is the {\em parallel set} of $\gamma$, i.e., 
the union of all geodesics which are a bounded distance from $\gamma$.

We fix an apartment $\A$ in ${\mathcal B}_M$ containing $o$; then $\A$ necessarily contains $F$ (and, hence, $\gamma$) as well. 
Let $\Delta_M$ be a chamber of ${\mathcal B}_M$ in $\A$ with tip $o$; by our convention set in section \ref{general}, 
$\De_M$  contains $\De$, the chamber of ${\mathcal B}$ with the tip $o$. We note that $\De_M$ splits as $F\times \De'_M$.  
 
In the case when ${\mathcal B}$ has rank one (i.e., is a tree) it is well-known that the correct notion of the ``distance'' from a point 
$x \in {\mathcal B}$ to a point $\xi \in \tits{\mathcal B}$ is the value of the Busemann function
$b_{\xi}(x)$, normalized to be zero at the base-point $o$.  
The usual Busemann function $b_{\xi}(x)$ is defined as follows:   
$$ 
b_{\xi}(x): = \lim_{t \to \infty}[ d(\ga(t),x) - d(\ga(t), o)].$$
We note that 
$$ 
b_{\xi}(o)=0.$$

In what follows we will define a $\De_M$-valued Busemann function $b_{\xi,\Delta_M}(x)$ (which should be considered as the $\De_M$-valued
``distance'' from $\xi$ to $x$). Again, we note that $b_{\xi,\Delta_M}$ will be defined so that 
$$ 
b_{\xi,\Delta_M}(o) = 0.$$

\subsection{The $\Delta_M$-valued distance function $\widetilde{d}_{\Delta_M}$ on ${\mathcal B}$}

Before defining vector-valued Busemann function, we first introduce a (partially defined) 
distance function $\widetilde{d}_{\Delta_M}$ on ${\mathcal B}$ which extends the function 
$d_{\De_M}$ defined on ${\mathcal B}_M$. We first note that for every $x\in {\mathcal B}$ there exists an apartment 
$\A_\xi\subset {\mathcal B}$ containing $x$, so that $\xi\in \tits \A_\xi$. For every two apartments $\A_\xi, \A'_\xi$ asymptotic to $\xi$ 
there exists a (typically non-unique) isomorphism $\phi_{\A_\xi,\A_\xi'}: \A_{\xi}\to \A_{\xi}'$ which fixes the intersection $\A_\xi\cap \A_{\xi}'$ 
pointwise. Since this intersection contains a ray asymptotic to $\xi$, it follows that the extension of   
$\phi_{\A_\xi,\A_\xi'}$ to the ideal boundary sphere of $\A_\xi$ fixes the point $\xi$. Thus, the apartments 
$\{\A_\xi\}$ form an {\em atlas} on ${\mathcal B}$ with transition maps in $\widetilde{W}_M$, the stabilizer of $\xi$ in $\widetilde{W}$. 
(Here we are abusing the notation and identify an apartment $\A_\xi$ in ${\mathcal B}$ and its isometric parameterization $\A\to \A_\xi$.) 
The only axiom of a building lacking in this definition is that not every two points in ${\mathcal B}$ belong to a common apartment $\A_\xi$. 
Nevertheless, for points $x,y \in {\mathcal B}$ which belong to some $\A_\xi$ we can repeat the definition of a chamber-valued distance 
function \cite{KLM1}: 
$$
\widetilde{d}_{\Delta_M}(x,y) = d_{\Delta_M}(\phi_{\A_{\xi}, \A}(x), \phi_{\A_{\xi}, \A}(y)).$$

In particular, $\widetilde{d}_{\Delta_M}$ restricted to $\B_M$ coincides with $d_{\De_M}$. 
Then $\widetilde{d}_{\Delta_M}$ is a partially-defined ${\Delta_M}$-valued distance function on ${\mathcal B}$. 
As in the case of the definition of $\De$-distance on ${\mathcal B}$, one verifies that $\widetilde{d}_{\Delta_M}$ is independent of the 
choices of apartments and their isomorphisms. It is clear from the definition that $\widetilde{d}_{\Delta_M}$ is invariant under 
the subgroup $P\subset G$ (but not under $G$ itself). 

\subsection{The Busemann function $b_{\xi,\Delta_M}(x)$}

We are now in position to repeat the definition of the usual Busemann function. 

Pick $x\in \B$ and let $\gamma'$ be a complete geodesic in ${\mathcal B}$ containing $x$ and asymptotic to $\xi$. 
Let ${\mathcal P}(\gamma')$ be the parallel set of $\gamma'$. The following simple lemma is critical in what follows.

\begin{lemma} \label{d-tilde}
There exists $t_0$ such that for all $t \geq t_0$ we have 
$\ga(t) \in {\mathcal P}(\gamma')$ and, furthermore,  the subray $\ga([t_0,\infty))$ and the point 
$x$ belong to a common apartment $\A'\subset {\mathcal P}(\ga')$. In particular, $\A'$ contains $\xi$ in its ideal boundary 
and $\widetilde{d}_{\Delta_M}(\ga(t),x)$ is well-defined for $t \geq t_0$.  
\end{lemma}
\begin{proof} We will give a proof only in the case when ${\mathcal B}=\B_G$ is the Bruhat-Tits building of a 
reductive group $G$ over a nonarchimedean valued field, although the statement holds for general Euclidean buildings as well. 

There exists a unipotent element $u\in G$ that carries ${\mathcal P}(\gamma)$ to ${\mathcal P}(\gamma')$ and fixes $\xi$.  
Since $u$ is unipotent and fixes $\xi$, it also fixes an infinite subray of $\overline{o\xi}$. So there exists $t_0$ such that 
$\ga([t_0,\infty)) \in {\mathcal P}(\gamma')$. Now choose an
apartment $\A'$ in ${\mathcal P}(\gamma')$ which contains $x$ and $\ga(t_0)$. Since every apartment in
${\mathcal P}(\gamma')$ contains $\gamma'$, we have $\xi \in \tits \A'$ and, consequently,  
$\ga(t) \in \A'$ for $t \geq t_0$. \end{proof} 

We  now define $b_{\xi,\Delta_M}(x)$ by
$$
b_{\xi,\Delta_M}(x):= \lim_{t \to \infty}[ \widetilde{d}_{\Delta_M}(\ga(t),x) - \widetilde{d}_{\Delta_M}(\ga(t), o)].
$$
We need to show that the limit on the right-hand side exists. 
We first claim that for each $t\ge t_0$, the difference vector 
$$
f(t,x):=\widetilde{d}_{\Delta_M}(\ga(t),x) - \widetilde{d}_{\Delta_M}(\ga(t), o)$$
is a well-defined element of $\Delta_M$, where $t_0$ is as in Lemma \ref{d-tilde} above. Indeed, 
by that lemma, $\widetilde{d}_{\Delta_M}(\ga(t),x)\in \De_M$ is well-defined for $t\ge t_0$. Next, 
$\ga\subset F$ and $\De_M=F\times \De'_M$. Therefore, 
$$ 
- \widetilde{d}_{\Delta_M}(\ga(t), o)= -\overrightarrow{\ga(t)o} \in \Delta_M.$$
Hence, by convexity of $\De_M$, $f(t,x)\in \Delta_M$.  

\begin{lemma}
$f(t,x)$ is constant in $t$ for $t \geq t_0$. 
\end{lemma}
\begin{proof}   
Let $\phi_{\A',\A}: \A' \to \A$ be an isomorphism of apartments as above fixing $\xi$. Hence 
$\phi_{\A', \A}(\ga(t)) = \ga(t)$ for $t \geq t_0$ and {\em by definition}
$$ 
\widetilde{d}_{\Delta_M}(\ga(t),x) = d_{\Delta_M}(\ga(t), x'), \,\, t \geq t_0,
$$
where $x':= \phi_{\A', \A}(x)$. Then
$$
f(t,x)= f(t,x')= d_{\Delta_M}(\ga(t),x') - d_{\Delta_M}(\ga(t), o),\,\, t \geq t_0. 
$$
Accordingly, replacing $x$ by $x'$ we have reduced to the case where $x', o$ and $\ga(t), t \geq t_0$ are contained
in the  apartment $\A$. Now apply an element of the Weyl group of $M$ (which fixes $\gamma$
and, hence, $\ga(t), \forall t$) to $x'$ to obtain $x'' \in \Delta_M$, so that
$$
\overrightarrow{\ga(t) x''}= \widetilde{d}_{\De_M}(\ga(t),x), \quad t\ge t_0. 
$$
The lemma now follows from 
$$
f(t,x') = \overrightarrow{\ga(t)x''}- \overrightarrow{\ga(t)o} = \overrightarrow{ox''},$$
see figure below. 

\begin{figure}[h]
\begin{displaymath}
\begin{xy}
(-40,0)*{}="A"; (0, 0)*{}="B"; (40, 0)*{}="C"; (20, 20)*{}="D";
"A"*{\bullet}; "B"*{\bullet}; "C"*{\bullet}; "D"*{\bullet};
{\ar "A";"B"};
{\ar "B"; "C"};
{\ar "B"; "D"};
{\ar "C"; "D"};
(-42, 2)*{  \xi}; (-6, 3)*{  \ga(t)}; (42, 2)*{  o}; (20, 23)*{  x''}
\end{xy}
\end{displaymath}
\caption{$f(t,x^{\prime})$ is constant for $t \geq t_0.$  }
\end{figure}
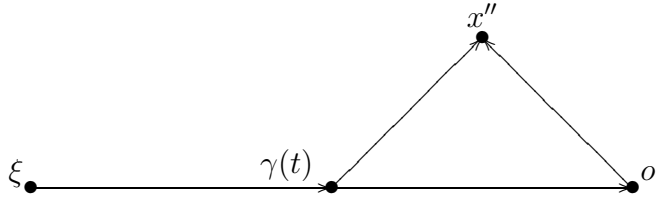

\end{proof}


\begin{lemma} \label{Busemann_comp_lem}
1. $b_{\xi,\De_M}$ is invariant under $K_P=N K_M$. 

2. If 
$b_{\xi,\De_M}(x)=b_{\xi,\De_M}(y)$ then $K_P x= K_P y$. 
\end{lemma}
\proof 1. Let $u\in P$ be unipotent. Then $u$ not only fixes $\xi$ but, moreover, for 
every geodesic ray $\beta: \R_+\to {\mathcal B}$ asymptotic to  $\xi$, there exists $t_u$ so that 
$u(\beta(t))=\beta(t)$ for all $t\ge t_u$. Therefore, by the invariance property of $\widetilde{d}_{\De_{M}}$,
$$
\widetilde{d}_{\De_{M}}(x, \ga(t))=\widetilde{d}_{\De_{M}}(u(x), u(\ga(t)))= \widetilde{d}_{\De_{M}}(u(x), \ga(t)). 
$$
It then follows from the definition of $b_{\xi,\De_M}$ that it is invariant under $u$. The same argument works 
for $u$ replaced with $k\in K_M$ since it suffices to know that (the entire ray) $\rho$ is fixed by $k$. 

\medskip 
2. For every $z\in \B$ there exists $n\in N$ so that $n(z)\in \B_M$. 
Therefore, applying elements of $N$ to $x, y$, we 
reduce the problem to the case when $x, y\in {\mathcal B}_M$. For such $x, y$,   
$$
d_{\De_M}(o, y)= b_{\xi,\De_M}(y)= b_{\xi,\De_M}(x)= d_{\De_M}(o, x). 
$$
Therefore, $x, y$ belong to the same $K_M$-orbit. \qed 

\medskip

This lemma allows us to give a purely algebraic characterization of the space of 
based ideal triangles $\I\T(\la,\mu; \xi)$: 

\begin{cor} \label{IT_comp_cor}
$
\I\T(\la,\mu; \xi)= K_P x_\la \cap K x_\mu = b_{\xi,\De_M}^{-1}(\la) \cap S_{\mu}(o).  
$
\end{cor}

\section{Retractions} \label{retractions_sec}

\subsection{The retractions $\rho_{{\bf b},\mathcal A}$} \label{rho_review}

Attached to any alcove ${\bf b}$ in an apartment $\mathcal A$ is a retraction $\rho_{{\bf b}, \mathcal A} : \B_G \rightarrow \mathcal A$.  It is distance-preserving and simplicial.   Let us abbreviate it here by $\rho$. Recall that 
this retraction is defined as follows: Pick an apartment $\A'\subset \B$ containing alcoves ${\bf b}$ and ${\bf x}$. Then 
there exists a unique isomorphism of apartments $\phi: \A' \to A$ fixing ${\bf b}$.  Then $\rho({\bf x}) := \phi({\bf x})$. 

We need to review some of the basic properties of $\rho$.  Let ${\bf x}$ denote any alcove in the building.  
Take a minimal gallery joining the base alcove ${\bf a} \subset \mathcal A$ to ${\bf x}$.   Let ${\bf x}_0$ denote the next-to-last alcove in this gallery.  Let $F_0$ denote the codimension one facet separating ${\bf x}_0$ from ${\bf x}$.   

Let ${\bf c}' := \rho({\bf x}_0)$, and let $H$ denote the hyperplane in $\mathcal A$ which contains $\rho(F_0)$ and let $s_{H}$ denote the corresponding reflection in $\mathcal A$.    Let ${\bf c} := s_H({\bf c}')$.  Assuming we know ${\bf c}'$ by induction, what are the possibilities for $\rho({\bf x})$?  To visualize this, we will imagine ${\bf x}_0$ and $F_0$ as being fixed, and ${\bf x}$ as ranging over the affine line ${\mathbb A}^1$ consisting of the set of all alcoves ${\bf x} \neq {\bf x}_0$ containing $F_0$ as a face. Then one of the following holds: 

\smallskip

\noindent ${\bf Position \ 1:}$ If ${\bf b}$ and ${\bf c}'$ are on the same side of $H$, then all ${\bf x} \in {\mathbb A}^1$ retract under $\rho$ onto ${\bf c}$;

\smallskip

\noindent ${\bf Position \ 2:}$ if ${\bf b}$ and ${\bf c}'$ are on opposite sides of $H$, then one point in ${\mathbb A}^1$ retracts onto ${\bf c}$, and all remaining points of ${\mathbb A}^1$ retract onto ${\bf c}'$.

\medskip
Suppose ${\bf x} \in \B_G$ is joined to the base alcove ${\bf a}$ by a gallery corresponding to a reduced word expression for $w \in \widetilde{W}$.  It follows from the discussion above that $\rho_{\bf b, \mathcal A}({\bf x}) \leq w{\bf a}$ with respect to the Bruhat order $\leq$ on the set of alcoves in $\mathcal A$.  
(This Bruhat order is determined by the set of simple affine reflections corresponding to the walls of ${\bf a}$.)

Using this, it is not difficult to prove the following statement.

\begin{lemma} \label{rho(x)_in_Omega(mu)}
Let $x \in \overline{Kx_\mu}$.  The $\rho_{{\bf b}, \mathcal A}(x) \in \Omega(\mu)$.
\end{lemma}

Finally, the following lemma is obvious.

\begin{lemma} \label{rho_conj1}
For any $w \in \widetilde{W}$, we have
\begin{equation} \label{w(rho)}
w \circ \rho_{{\bf b}, \mathcal A} \circ w^{-1} = \rho_{w{\bf b}, \mathcal A}.
\end{equation}
\end{lemma}

\subsection{The retractions $\rho_{-\nu, \Delta_G-\nu}$}

For any $\nu \in X_*(T)$, let $W_{-\nu}$ denote the finite Weyl group at $x_{-\nu} \in \mathcal A$, namely, the group generated by the affine reflections in $\mathcal A$ which fix the point $x_{-\nu}$.  Regard $\Delta_G - \nu$ as the $G$-dominant Weyl chamber with apex $x_{-\nu}$.  Then consider the retraction
\begin{align} \label{rho_Delta_G}
\rho_{\Delta_G-\nu}:\mathcal A &\rightarrow \Delta_G - \nu \\
v &\mapsto w(v) \notag
\end{align}
where $w \in W_{-\nu}$ is chosen so that $w(v) \in \Delta_G - \nu$.  

Next choose any alcove ${\bf b} \subset \mathcal A$ with vertex $x_{-\nu}$.  Then the composition
$$
\rho_{\Delta_G-\nu} \circ \rho_{{\bf b}, \mathcal A}: \B_G \rightarrow \Delta_G - \nu
$$
is a retraction of the building onto the chamber $\Delta_G - \nu$.  Because of (\ref{w(rho)}), it is independent of the choice of ${\bf b}$.  Hence we may set
\begin{equation} \label{rho_nu_def}
\rho_{-\nu, \Delta_G-\nu} :=  \rho_{\Delta_G-\nu} \circ \rho_{{\bf b}, \mathcal A}.
\end{equation}

The following lemma will be useful later.

\begin{lemma} \label{preimage1}
For any $G$-dominant cocharacter $\mu$ such that $\nu + \mu$ is also $G$-dominant, we have
$$
\rho^{-1}_{-\nu, \Delta_G-\nu}(x_\mu) = (t^{-\nu}Kt^\nu)x_\mu.
$$
\end{lemma}

\begin{proof}
The special case
\begin{equation} \label{spec_case}
\rho^{-1}_{0, \Delta_G}(x_{\nu + \mu}) = Kx_{\nu+\mu}
\end{equation}
is obvious.  We have the identities (cf. Lemma \ref{rho_conj1})
\begin{align*}
t^{-\nu} \circ \rho_{{\bf b},\mathcal A} \circ t^\nu &= \rho_{{\bf b}-\nu, \mathcal A} \\ 
t^{-\nu} \circ \rho_{\Delta_G} \circ t^\nu &= \rho_{\Delta_G - \nu},
\end{align*}
and these together with the definition (\ref{rho_nu_def}) yield the formula
\begin{equation} \label{rho_conj2}
t^{-\nu} \circ \rho_{0,\Delta_G} \circ t^{\nu} = \rho_{-\nu, \Delta_G-\nu}.
\end{equation}
Now the desired formula follows from the special case (\ref{spec_case}) above.
\end{proof}

\subsection{The retraction $\rho_{K_P,\Delta_M}$} \label{rho_K_P_Delta_M}

Recall that for a Borel subgroup $B = TU$, there is a corresponding retraction
$$
\rho_{U,\mathcal A}: \B_G \rightarrow \mathcal A
$$
which can be realized as $\rho_{{\bf b}, \mathcal A}$ for ${\bf b}$ ``sufficiently antidominant'' with respect to the roots in ${\rm Lie}(U)$.  We also have the retraction
$$
\rho_{K,\Delta_G} := \rho_{0, \Delta_G}: \B_G \rightarrow \Delta_G
$$
discussed above.  We want to define a retraction
$$
\rho_{K_P,\Delta_M} : \B_G \rightarrow \Delta_M
$$
which interpolates between these two extremes, $\rho_{U,\mathcal A}$ and $\rho_{K,\Delta_G}$.

Before defining $\rho_{K_P,\Delta_M}$ we shall review the construction of a similar retraction $\rho_{I_P,\mathcal A}$ which was introduced in \cite{GHKR2}.

Consider the following two properties of an element $\nu \in X_*(T)$:
\begin{enumerate}
\item[(i)] $\langle \alpha, \nu \rangle = 0 $ for all roots $\alpha \in \Phi_M$;
\item[(ii)] $\langle \alpha, \nu \rangle >>0 $ for all roots $\alpha \in \Phi_N$.
\end{enumerate}
We say $\nu$ is {\em $M$-central} if (i) holds and  and {\em very $N$-dominant} if (ii) holds. 

Let ${\bf a}_M$ denote the base alcove in $\mathcal A$ determined by some set of positive roots $\Psi^+_M$ in $M$.  (For example, we could take $\Psi^+_M =  \Phi_M^+ := \Phi^+ \cap \Phi_M$.)  This means that ${\bf a}_M$ is the region of $\mathcal A$ which lies between the hyperplanes $H_\alpha$ and $H_{\alpha -1}$ for every $\alpha \in \Psi^+_M$.  Let $I_M := I \cap M$, the Iwahori subgroup of $M$ which corresponds to the alcove ${\bf a}_M$.  Set $I_P := N \cdot I_M$.

Let $\mathcal S$ be any bounded subset of the building $\B_G$.   
\begin{lemma} [\cite{GHKR2}] \label{rho_b_A_indep}
Consider the following properties of an alcove ${\bf b} \subset \mathcal A$: 
\begin{enumerate}
\item[(a)] ${\bf b} \subset {\bf a}_M$;
\item[(b)] ${\bf b}$ has a vertex $x_{-\nu}$, where $\nu$ is $M$-central and very $N$-dominant (depending on $\mathcal S$).
\end{enumerate}  
Then the corresponding retractions $\rho_{{\bf b}, \mathcal A}$, for ${\bf b}$ satisfying (a) and (b), all agree on the set $\mathcal S$.  
\end{lemma}
\begin{proof}
By \cite{GHKR2}, Lemma 11.2.1, there is a decomposition 
$$G = \coprod_{w \in \widetilde{W}} I_P\, w \, I.
$$
It follows that $\B_G$ is the union of all translates $g^{-1}\mathcal A$, as $g$ ranges over $I_P$.  Moreover, for $g \in I_P$, the map $g: g^{-1}\mathcal A \rightarrow \mathcal A$ is a simplicial map fixing the alcove ${\bf b}$ as long as ${\bf b} \subset {\bf a}_M$ and ${\bf b}$ is sufficiently antidominant  with respect to $N$; it follows that, on $g^{-1}\mathcal A$, the map $g: g^{-1}\mathcal A \rightarrow A$ coincides with $\rho_{{\bf b}, \mathcal A}$ for such ${\bf b}$.  

To be more precise, for a bounded subset $\mathcal S \subset \B_G$ let us give a condition $(*_\mathcal S)$ on alcoves ${\bf b}$ such that all retractions $\rho_{\bf b, \mathcal A}$ for ${\bf b}$ satisfying $(*_\mathcal S)$ will coincide on $\mathcal S$.  There exists a bounded subgroup $I_\mathcal S \subseteq I_P$ such that $\mathcal S \subset \cup_{g \in I_\mathcal S} g^{-1}\A$.  Let $I_{\bf b}$ denote the Iwahori subgroup fixing ${\bf b}$.  Then condition $(*_\mathcal S)$ can be taken to be
$$
\noindent (*_\mathcal S) \hspace{.5in} I_\mathcal S \subset I_{\bf b}.
$$
This condition suffices, because if ${\bf b}$ satisfies $(*_\mathcal S)$, then $\rho_{\bf b, \mathcal A}$ and $g \in I_\mathcal S$ will coincide as maps $g^{-1}\mathcal A \to \mathcal A$, since both are simplicial and fix ${\bf b}$.  This proves the lemma.  
\end{proof}

Denote by $\rho_{I_P,\mathcal A}(v)$ the common value of all retractions $\rho_{{\bf b}, \mathcal A}(v)$ where ${\bf b}$ ranges over a set (depending on $v$) of alcoves as in the lemma.  This defines a retraction
$$\rho_{I_P,\mathcal A}: \B_G \rightarrow \mathcal A.$$  
As the notation indicates, it depends on the choice of the alcove ${\bf a}_M$ and the parabolic $P = MN$.  The following result appeared in \cite{GHKR2}.  We give a proof for the benefit of the reader.

\begin{lemma} [\cite{GHKR2}]  \label{GHKR2} Suppose $I_P$ is defined using $P$ and ${\bf a}_M$ as above.  Then:
\begin{enumerate}
\item[(i)] For any $g \in I_P$, we have $\rho_{I_P,\mathcal A}|_{g^{-1}\mathcal A} = g$.
\item[(ii)] For any alcove ${\bf x}$ in $\mathcal A$, we have $\rho^{-1}_{I_P, \mathcal A}({\bf x}) = I_P\, {\bf x}$.
\end{enumerate}  
\end{lemma}

\begin{proof}
Part (i) follows from the proof of the lemma above.  It implies that $\rho_{I_P,\mathcal A}$ is $I_P$-invariant, in the following sense: thinking of $\rho_{I_P,\mathcal A}$ as a map $G/I \rightarrow \widetilde{W}I/I$ sending alcoves in $\B_G$ to those in $\mathcal A$, for each $w \in \widetilde{W}$, we have
$$
\rho_{I_P,\mathcal A}^{-1}(w {\bf a}) \supset I_P w {\bf a}.
$$
But since we have disjoint decompositions
$$
G/I = \coprod_{w \in \widetilde{W}} I_P w I/I = \coprod_{w \in \widetilde{W}} \rho^{-1}_{I_P,\mathcal A}(w {\bf a}),
$$
this containment must actually be an equality.  This proves part (ii).
\end{proof}

We now turn to the variant of interest to us here, namely the retraction $\rho_{K_P, \Delta_M}$ onto the $M$-dominant Weyl chamber $\Delta_M \subset \mathcal A$ .  Here we recall $K_M = K \cap M$, and 
$K_P = N \cdot K_M$; 

\begin{defn}
We define the retraction $\rho_{K_P,\De_M}: \B_G \rightarrow \De_M$ by setting
$$
\rho_{K_P,\De_M} = \rho_{\De_M} \circ \rho_{I_P,\A}.
$$
Here $I_P$ is determined by $P = MN$ and the $M$-alcove ${\bf a}_M$ defined in terms of some set of positive roots $\Psi^+_M$ for $M$.
\end{defn}
Of course, we need to show that this is indeed independent of the choice of ${\bf a}_M$.  

\begin{lemma} \label{rho_K_P_defined}
The retraction $\rho_{\De_M} \circ \rho_{I_P,\A}$ is independent of the choice of $M$-alcove ${\bf a}_M$.
\end{lemma}

\begin{proof}
Fix any bounded set $\mathcal S \subset \B_G$. Recall that on $\mathcal S$, $\rho_{I_P,\A}$ can be realized as the retraction $\rho_{\bf b - \nu, \A}$, for any $M$-central and very $N$-dominant cocharacter $\nu$ (depending on $\mathcal S$) and any alcove $\bf b$ whose closure contains $o$ and which is contained in ${\bf a}_M$.  Suppose ${\bf a}^1_M$ and ${\bf a}^2_M$ are two $M$-alcoves, and ${\bf b}_i \subset {\bf a}^i_M$ are two such alcoves.  We need to show that $\rho_{\De_M} \circ \rho_{\bf b_1-\nu,\A} = \rho_{\De_M} \circ \rho_{\bf b_2-\nu,\A}$ on $\mathcal S$.  

Let $w \in W_M$ be such that $w{\bf a}^1_M = {\bf a}^2_M$.  Then by Lemma \ref{rho_b_A_indep}, we know already that $\rho_{w\bf b_1-\nu,\A} = \rho_{\bf b_2-\nu,\A}$ on $\mathcal S$.  So, it remains to see that $\rho_{\De_M} \circ \rho_{\bf b_1 - \nu, \A} = \rho_{\De_M} \circ \rho_{w{\bf b}_1-\nu,\A}$ on $\mathcal S$.

Without loss of generality, we may assume $w = s_\alpha$, the reflection corresponding to a simple root $\alpha$ in the positive system $\Phi^+_M$.  We claim that $\rho_{w\bf b_1-\nu,\A}({\bf x}) \in \{ \rho_{\bf b_1-\nu,\A}({\bf x}), w(\rho_{\bf b_1-\nu,\A}({\bf x})) \}$.  This will imply the lemma.

Denote by $F$ the unique codimension 1 facet in $\A$ which separates ${\bf c}_1 := {\bf b}_1-\nu$ from ${\bf c}_2 := w{\bf b}_1 -\nu$.  Abbreviate $\rho_{\bf c_i,\A}$ by $\rho_i$ for $i=1,2$.

If ${\bf x} \subset \A$, then both retractions just give ${\bf x}$ and there is nothing to do.  So, assume ${\bf x}$ is not contained in $\A$.  Choose a minimal gallery ${\bf x}\!=\!{\bf x}_0, {\bf x}_1, \dots, {\bf x}_{d}, F$ joining ${\bf x}$ to $F$.  The notation means $F$ is a face of ${\bf x}_d$ and $d \geq 0$ is the distance $d({\bf x},F)$ (by definition, for a facet $\mathcal F$, the distance $d({\bf x},\mathcal F)$ is the minimal number of wall-crossings needed to form a gallery from  ${\bf x}$ to an alcove ${\bf y}$ with $\mathcal F \subset \overline{\bf y}$).

There are two cases to consider.  Assume first that ${\bf x}_d$ is one of the ${\bf c}_i$.  Let us assume ${\bf x}_d = {\bf c}_1$ (the other case is similar).  Then $d({\bf x}, {\bf c}_1) = d$ and $d({\bf x},{\bf c}_2) = d+1$.  In this case ${\bf c}_2, {\bf c}_1, {\bf x}_{d-1}, \dots, {\bf x}$ is a minimal gallery joining ${\bf c}_2$ (and also ${\bf c}_1$) to ${\bf x}$, and it follows that $\rho_1({\bf x}) = \rho_2({\bf x})$.

Now assume that ${\bf x}_d$ is neither ${\bf c}_1$ nor ${\bf c}_2$.  Then there are two minimal galleries
$$
{\bf c}_1, {\bf x}_d, \dots, {\bf x} \hspace{1in} {\bf c}_2, {\bf x}_d, \dots, {\bf x}
$$
and ${\bf x}_d \nsubseteq \A$ (and hence no ${\bf x}_i$ lies in $\A$; see e.g.~\cite{BT1}, (2.3.6)).  Thus these galleries leave $\A$ at $F$, and by looking at how they fold down onto $\A$ under $\rho_1$ and $\rho_2$, we see that $\rho_1({\bf x}) = w(\rho_2({\bf x}))$.

This proves the claim, and thus the lemma.
\end{proof}

Lemma \ref{GHKR2}(ii) above has the following counterpart.  Using Lemma \ref{Busemann_comp_lem}, we see from it that $\rho_{K_P,\De_M} = b_{\xi,\De_M}$.

\begin{lemma} \label{preimage2}
For any $M$-dominant $\lambda \in X_*(T)$,  we have
$$
\rho^{-1}_{K_P, \Delta_M}(x_\lambda) = K_P x_\lambda.
$$
\end{lemma}
\begin{proof}
First we show that $\rho_{K_P,\De_M}$ is $K_P$-invariant; this will show $\rho^{-1}_{K_P,\De_M}(x_\lambda) \supseteq K_P x_\lambda$.  Choose any $M$-alcove ${\bf a}_M$ and corresponding Iwahori $I_M$ as in the definition of $\rho_{K_P,\De_M}$ as $\rho_{\De_M} \circ \rho_{I_P,\A}$.  We can write $K_P = K_M N = I_M W_M I_M N$.  By Lemma \ref{GHKR2}(ii), $\rho_{I_P,\A}$ is $I_P$-invariant, and so it suffices to show that $\rho_{K_P,\De_M}$ is $W_M$-invariant.  On a bounded set $\mathcal S$, we realize this retraction as $\rho_{\De_M} \circ \rho_{\bf b, \A}$ for some alcove ${\bf b}$ contained in ${\bf a}_M$ and sufficiently anti-dominant with respect to the roots in $\Phi_N$.  Now for $w \in W_M$ and $v \in \mathcal S$, we have 
$$
\rho_{\De_M}\circ \rho_{\bf b,\A}(wv) = \rho_{\De_M}\big(w(\rho_{w^{-1}\bf b,\A}(v)) = \rho_{\De_M} \circ \rho_{w^{-1}\bf b,\A}(v)
$$
using Lemma \ref{rho_conj1} for the first equality.  But by Lemma \ref{rho_K_P_defined}, the right hand side is $\rho_{K_P,\De_M}(v)$, and the invariance is proved.

Next we show the opposite inclusion.  We have
$$
\rho^{-1}_{K_P,\De_M}(x_\lambda) = \cup_{w \in W_M} \rho^{-1}_{I_P,\A}(wx_\lambda).
$$
By Lemma \ref{GHKR2}(ii), the right hand side is contained in $I_PW_M x_\lambda$, which certainly belongs to $K_Px_\lambda$.  This completes the proof.
\end{proof}

We deduce the following interpolation between the Cartan and Iwasawa decompositions of $G$:

\begin{cor}
The map $\widetilde{W} \rightarrow G$ induces a bijection
$$
W_M\backslash \Lambda = W_M \backslash \widetilde{W}/W ~ \widetilde{\rightarrow} ~ K_P\backslash G/K.
$$ 
\end{cor}

\medskip 
The next lemma is a rough comparison between the retractions $\rho_{K_P, \Delta_M}$ and $\rho_{-\nu, \Delta_G-\nu}$.  (Actually it concerns their restrictions to a given bounded subset $\mathcal S \subset \B_G$.)  It will be made much more precise in the next section.

\begin{lemma} \label{comparison1}
For any bounded set $\mathcal S \subset \B_G$, there are elements $\nu$ which are $M$-central and very $N$-dominant 
(depending on $\mathcal S$) such that
$$
\rho_{K_P, \Delta_M}|{\mathcal S} = \rho_{-\nu, \Delta_G-\nu}|{\mathcal S}.
$$
\end{lemma}

\begin{proof}
First we remark that for $\nu$ which is $M$-central and very $N$-dominant, and for any alcove ${\bf b}$ having 
$x_{-\nu}$ as a vertex, every point $x \in \rho_{{\bf b}, \mathcal A}(\mathcal S)$ satisfies
$$
\langle \alpha, \nu + x \rangle > 0
$$
for $\alpha \in \Phi_N$.  Thus the retraction of such an element $x$ into $\Delta_G - \nu$ coincides with its retraction into $\Delta_M$ (and both are achieved by applying a suitable element of $W_M$).  The lemma follows from this remark and the definitions. 
\end{proof}

\section{Sharp comparison of $\rho_{K_P, \Delta_M}$ and $\rho_{-\nu, \Delta_G-\nu}$}  \label{sharp_sec}

\subsection{Statement of key proposition}

The following is the key technical device of this paper.  It is a much sharper version of Lemma \ref{comparison1}.

\begin{prop}  \label{key_prop}  Let $\mu$ be a $G$-dominant element of $X_*(T)$.  
Suppose $\nu \geq^P \mu$.  Then 
\begin{equation}
\rho_{-\nu, \Delta_G-\nu}|{\overline{K x_\mu}} = \rho_{K_P, \Delta_M}|{\overline{Kx_\mu}}.
\end{equation}
\end{prop}

The first lemma deals with the images of the two retractions appearing in Proposition \ref{key_prop}.

\begin{lemma} \label{image} Assume $\nu \geq^P \mu$.  Then the following statements hold.
\begin{enumerate}
\item[(a)] We have $\Omega(\mu) \cap \Delta_M = \Omega(\mu) \cap (\Delta_G - \nu)$.
\item[(b)] For $\lambda \in \Delta_M$, the intersection $K_P x_\lambda \cap \overline{Kx_\mu}$ is nonempty only if 
$\la\in \Omega(\mu)$.
\item[(c)] For $\lambda \in \Delta_G-\nu$, the intersection $(t^{-\nu}Kt^\nu)x_\lambda \cap \overline{Kx_\mu}$ is nonempty only if $\lambda \in \Omega(\mu)$.
\end{enumerate}
\end{lemma}

\begin{proof}
Part (a) follows easily from the definitions.  Part (b) is well-known (cf. \cite{GHKR1}, Lemma 5.4.1).  Let us prove (c).  Assume $\lambda \in \Delta_G -\nu$ makes the intersection in (c) nonempty.  By Lemma \ref{preimage1}, we see that $x_\lambda$ lies in $\rho_{-\nu, \Delta_G-\nu}(\overline{K\mu})$.  By Lemma \ref{rho(x)_in_Omega(mu)} and the definition of $\rho_{-\nu, \Delta_G - \nu}$, there exists $w_{-\nu} \in W_{-\nu}$ such that $w_{-\nu}(\lambda) \in \Omega(\mu)$.  Write $w_{-\nu} = t^{-\nu}w t^\nu$ for some $w \in W$.   Then we see that 
\begin{equation} \label{w(lambda+nu)}
w(\lambda + \nu) = w(\tilde{\lambda}) + \nu
\end{equation}
for some $\tilde{\lambda} \in \Omega(\mu)$.  The equation (\ref{w(lambda+nu)}) has the same form if we apply any element $w' \in W_M$ to it.  By replacing $w$ with a suitable element of the form $w'w$ ($w' \in W_M$), we may assume the right hand side of (\ref{w(lambda+nu)}) is M-dominant.  But then since $\nu \geq^P \mu$, the right hand side is $G$-dominant. Then $w(\lambda + \nu)$ and $\lambda + \nu$ are both $G$-dominant, hence they coincide.  This yields $\lambda = w(\tilde{\lambda})$; thus $\lambda$ belongs to $\Omega(\mu)$, as desired.    
\end{proof}

We will rephrase Proposition \ref{key_prop} using the following lemma.

\begin{lemma} \label{equiv1}
Fix a $G$-dominant cocharacter $\mu$.  Then the following are equivalent conditions on an element $\nu$ satisfying $\nu \geq^P \mu$:
\begin{enumerate}
\item[(i)] $(t^{-\nu}Kt^\nu)x_\lambda \cap \overline{Kx_\mu} = K_P x_\lambda \cap \overline{Kx_\mu}$, for all $\lambda \in \Omega(\mu) \cap \Delta_M$;
\item[(ii)] $\rho_{-\nu, \Delta_G-\nu}|{\overline{Kx_\mu}} = \rho_{K_P, \Delta_M}|{\overline{Kx_\mu}}$.
\end{enumerate}
\end{lemma}

\begin{proof}
This is immediate in view of Lemmas \ref{preimage1}, \ref{preimage2}, and \ref{image}.
\end{proof}

In a similar way, we have the following result.

\begin{lemma} \label{equiv2} The following are equivalent conditions on cocharacters $\nu_1$ and $\nu_2$ which satisfy $\nu_i \geq^P \mu$ for $i=1,2$:
\begin{enumerate}
\item[(a)] $(t^{-\nu_1}Kt^{\nu_1})x_{\lambda} \cap \overline{Kx_\mu} = (t^{-\nu_2}Kt^{\nu_2})x_{\lambda} \cap \overline{Kx_\mu}$ for all $\lambda \in \Omega(\mu) \cap \Delta_M$;
\item[(b)] $\rho_{-\nu_1, \Delta_G-\nu_1}|{\overline{Kx_\mu}} = \rho_{-\nu_2, \Delta_G-\nu_2}|{\overline{Kx_\mu}}$;
\item[(c)] $\rho_{{\bf b}_1,\mathcal A}(x) \in W_M(\rho_{{\bf b}_2,\mathcal A}(x))$, 
for $x \in \overline{Kx_\mu}$.
\end{enumerate}
Here for $i=1,2$, ${\bf b}_i$ is any alcove having $x_{-\nu_i}$ as a vertex. 
\end{lemma}
 As discussed above in the context of $\rho_{-\nu, \Delta_G-\nu}$, in proving (c) we are free to use any alcove having $x_{-\nu_i}$ as a vertex we wish.

\begin{proof}
The equivalence of (a) and (b) is clear.  The equivalence of (b) and (c) follows by the same argument which proved Lemma \ref{comparison1}.
\end{proof}

If (a),(b), or (c) hold, then so does (ii) in Lemma \ref{equiv1}, by virtue of Lemma \ref{comparison1}.  To see this,  in (b) above, take $\nu = \nu_1$ and take $\nu_2 \geq^P \mu$ sufficiently dominant with respect to $N$ so that, for $\mathcal S = \overline{Kx_\mu}$, the conclusion of Lemma \ref{comparison1} holds for it.

Thus, the following proposition will imply Proposition \ref{key_prop} (and in fact it is equivalent to Proposition \ref{key_prop}).

\begin{prop} \label{better_prop}
Suppose $\nu_i \geq^P \mu$ for $i = 1,2$.  Then
$$
\rho_{-\nu_1, \Delta_G-\nu_1}|{\overline{Kx_\mu}} = \rho_{-\nu_2, \Delta_G-\nu_2}|{\overline{Kx_\mu}}.
$$
\end{prop}

\subsection{Proof of Proposition \ref{better_prop}}

We will prove (c) in Lemma \ref{equiv2} holds.  We make particular choices for the ${\bf b}_i$, namely, we set
$$
{\bf b}_i := w_0{\bf a} - \nu_i
$$
for $i = 1,2$, where $w_0$ denotes the longest element in the Weyl group $W$ (so that $w_0{\bf a}$ is the alcove at the origin which is in the {\em anti-dominant} Weyl chamber).  Set $\rho_i = \rho_{{\bf b}_i, \mathcal A}$ for $i=1,2$.  For these choices of ${\bf b}_i$, we will prove the following more precise fact:
$$
\mbox{for $x \in \overline{Kx_\mu}$, we have $\rho_1(x) = \rho_2(x)$}.
$$
Let $x \in \overline{Kx_\mu}$ and let ${\bf a}\!=\!{\bf a}_0, {\bf a}_1, \dots, {\bf a}_{r-1}, {\bf a}_r$ be any minimal gallery in ${\mathcal B}_G$ joining ${\bf a}$ to $x$ (this means that $x$ belongs to the closure of ${\bf a}_r$ and $r$ is minimal with this property).  It is obviously enough for us to prove that
$$
\rho_1({\bf a}_r) = \rho_2({\bf a}_r).
$$
We will prove this by induction on $r$.  But first, we must formulate an alcove-theoretic proposition (Proposition \ref{alcove_prop} below).  This involves the notion of $\mu$-admissible alcove (cf.~\cite{HN}, \cite{KR}).  By definition, an alcove in $\mathcal A$ is $\mu$-{\em admissible} provided it can be written in the form $w{\bf a}$ for $w \in \widetilde{W}$ such that $w \leq t_\lambda$ for some $\lambda \in W\mu$.  Here $\leq$ is the Bruhat order on $\widetilde{W}$ determined by the alcove ${\bf a}$.

The set of $\mu$-admissible alcoves is closed under the Bruhat order on any given apartment: if ${\bf a}_r$ is $\mu$-admissible, and ${\bf a}_{r-1}$ precedes ${\bf a}_r$, then ${\bf a}_{r-1}$ is also $\mu$-admissible.  Moreover, if $x \in \Omega(\mu)$, then the minimal length alcove containing $x$ in its closure is always $\mu$-admissible.  These remarks imply that the next proposition suffices to prove Proposition \ref{better_prop}.

\begin{prop} \label{alcove_prop}
Suppose ${\bf a}\!=\!{\bf a_0},{\bf a}_1, \dots, {\bf a}_{r}$ is any minimal gallery in ${\mathcal B}_G$ such that, in an apartment $\mathcal A'$ containing this gallery, the terminal alcove ${\bf a}_r$ (and thus every other alcove ${\bf a}_i$) is a $\mu$-admissible alcove in $\mathcal A'$.   Then we have
$$\rho_1({\bf a}_r) = \rho_2({\bf a}_r).$$
\end{prop}

\begin{proof}
We proceed by induction on $r$.  There is nothing to prove for $r=0$.  Assume $r \geq 1$ and that $\rho_1({\bf a}_{r-1}) = \rho_2({\bf a}_{r-1})$.  In particular, if $F_r$ is the face separating ${\bf a}_{r-1}$ from ${\bf a}_r$, we have $\rho_1(F_r) = \rho_2(F_r)$.  Let $H \subset \mathcal A$ denote the hyperplane containing $\rho_1(F_r)$; write $s_H$ for the corresponding reflection in $\mathcal A$.  Set ${\bf c}' := \rho_1({\bf a}_{r-1})$ and ${\bf c} = s_H({\bf c}')$.  

We now recall the discussion of subsection \ref{rho_review}.  For $i = 1,2$, we have 
$$
\rho_i({\bf a}_r) \in  \{ {\bf c}', {\bf c} \}.
$$
Following subsection \ref{rho_review}, we need to show that the alcoves ${\bf b}_1$ and ${\bf b}_2$ are simultaneously in Position 1 (or 2) with respect to ${\bf c}'$ and $H$. 

We claim that ${\bf b}_1$ and ${\bf b}_2$ are both on the same side of $H$.  To see this, we may assume 
$$
H = H_{\alpha-k} = \{ p \in X_*(T)_\mathbb R ~ | ~ \langle \alpha, p \rangle = k \}.$$ 
for a positive root $\alpha$.  If $\alpha \in \Phi_M$, then our assertion follows from the fact that ${\bf b}_1$ and ${\bf b}_2$ belong to the same ``$M$-alcove'' (the anti-dominant one).  If $\alpha \in \Phi_N$, it follows because $\nu_i \geq^P \mu$ and because $H$ intersects the set ${\rm Conv}(W\mu)$.  Indeed, let $p$ be a point in ${\rm Conv}(W\mu) \cap H$.   Then the condition 
$$\langle \alpha, \nu_i + p \rangle \geq 0$$
 implies that  
$$
\langle \alpha, -\nu_i \rangle \leq k,
$$
which shows that $x_{-\nu_1}$ and $x_{-\nu_2}$ are on the same side of $H$, and this suffices to prove the claim.

This shows that the ${\bf b}_i$ are simultaneously in Position 1 (or 2) with respect to $H$ and ${\bf c}'$.  This is almost, but not quite enough by itself, to show that $\rho_1({\bf a}_r) = \rho_2({\bf a}_r)$.  Position 1 is quite easy to handle (see below).  As we shall see, in Position 2, ${\bf a}_r$ could retract under $\rho_i$ to either ${\bf c}$ or ${\bf c}'$;  we require an additional argument, given below, to show the two retractions must coincide.  

First assume we are in Position 1, that is, ${\bf b}_i$ and ${\bf c}'$ are on the same side of $H$.  Then as in subsection \ref{rho_review}, we see that $\rho_i({\bf a}_r) = {\bf c}$ for $i=1,2$.

Now assume we are in Position 2, that is, ${\bf b}_i$ and ${\bf c}'$ are on opposite sides of $H$.  In this case it is a priori possible that (say) $\rho_1({\bf a}_r) = {\bf c}$ while $\rho_2({\bf a}_r) = {\bf c}'$.  Our analysis below will show that this cannot happen.

Choose any minimal gallery joining ${\bf b}_1 = w_0{\bf a}-\nu_1$ to $F_r$:
$$
{\bf b}_1\! = \!{\bf x}_0, {\bf x}_1, \dots, {\bf x}_n
$$
where $F_r$ is contained in the closure of ${\bf x}_n$ (and $n$ is minimal with this property).  Similarly, choose a minimal gallery joining ${\bf b}_2$ to $F_r$.
$$
{\bf b}_2\!=\!{\bf y}_0, {\bf y}_1, \dots, {\bf y}_m.
$$
Further, set ${\bf b} := w_0{\bf a} - \nu_1 - \nu_2$ and fix two minimal galleries in ${\mathcal A}$
\begin{align*}
{\bf b}\!&=\!{\bf u}_0, {\bf u}_1, \dots, {\bf u}_p\!=\!{\bf b}_1 \\
 {\bf b}\!&=\!{\bf v}_0, {\bf v}_1, \dots, {\bf v}_q\!=\!{\bf b}_2.
\end{align*}
It is clear that $\rho_1$ (resp. $\rho_2$) fixes the former (resp. latter).  

\medskip

\noindent {\bf Claim:}  The concatenation ${\bf b}, \dots, {\bf b}_1, \dots, {\bf x}_n$ is a {\em minimal} gallery.  (The same proof will show that ${\bf b}, \dots, {\bf b}_2, \dots, {\bf y}_m$ is minimal.)   This may be checked after applying the retraction $\rho_1$.  But $\rho_1({\bf b}_1), \dots, \rho_1({\bf x}_n)$ is a minimal gallery joining ${\bf b}_1$ to an alcove of ${\mathcal A}$ contained in the convex hull ${\rm Conv}(W\mu)$, and consequently this gallery is in the $J_P$-{\em positive} direction (see subsection \ref{pos} below).  The gallery ${\bf b}, \dots, {\bf b}_1$ is also in the $J_P$-positive direction (and is fixed by $\rho_1$).  

It would be natural to hope that the concatenation of two galleries in the $J_P$-positive direction is always in the $J_P$-positive direction.  Then we could invoke Lemma \ref{J_P-pos} below to prove that ${\bf b}, \dots, {\bf b}_1\!=\!\rho_1({\bf b}_1) \dots, \rho_1({\bf x}_n)$ is minimal.  But in general it is {\em not} true that the 
concatenation of two galleries in the $J_P$-direction is also in the $J_P$-positive direction (think of the extreme case $M = G$).  However, in our situation, because ${\bf b}, \dots, {\bf b}_1$ lies entirely in the same $M$-alcove (the anti-dominant one), the concatenation ${\bf b}, \dots, {\bf b}_1\!=\!\rho_1({\bf b}_1), \dots \rho_1({\bf x}_n)$ is nevertheless in the $J_P$-positive direction.  Hence the concatenation ${\bf b}, \dots, {\bf b}_1, \dots, {\bf x}_n$ is minimal.  The claim is proved.

It follows that the concatenations
\begin{align*}
{\bf u}_0, &\dots, {\bf u}_p\!=\!{\bf x}_0, \dots, {\bf x}_n \\
{\bf v}_0, &\dots, {\bf v}_q\!=\!{\bf y}_0, \dots, {\bf y}_m
\end{align*}
are two minimal galleries joining ${\bf b}$ to $F_r$.  By a standard result (cf. \cite{BT1}, (2.3.6)), any two such galleries belong to any apartment that contains both ${\bf b}$ and $F_r$.  In particular, ${\bf x}_n = {\bf y}_m$.   
Now recall we are assuming that ${\bf c}'$, the common value of $\rho_1({\bf a}_{r-1})$ and $\rho_2({\bf a}_{r-1})$, is on the opposite side of $H$ from the alcoves ${\bf b}_i$, and so ${\bf c}$ is on the same side of $H$ as the ${\bf b}_i$.  On the other hand, $\rho_1({\bf x}_n)$ and $\rho_2({\bf y}_m)$ are both equal to ${\bf c}$, since that is the alcove sharing the facet $\rho_i(F_r)$ with ${\bf c}'$ but on the same side of $H$ as ${\bf b}_i$.  In particular, we have ${\bf a}_{r-1} \neq {\bf x}_n$.   

If ${\bf a}_r = {\bf x}_n = {\bf y}_m$, then we have $\rho_1({\bf a}_r) = {\bf c} = \rho_2({\bf a}_r)$.  If ${\bf a}_r \neq {\bf x}_n$, then we have $\rho_1({\bf a}_r) = {\bf c}' = \rho_2({\bf a}_r)$.  This completes the proof of Proposition \ref{alcove_prop}.
\end{proof}

\subsection{Galleries in the $J_P$-positive direction} \label{pos}

We define $J \subset G$ to the be the Iwahori subgroup corresponding to the ``anti-dominant'' alcove $w_0{\bf a}$.  Set 
$J_M := J \cap M$ and $J_P := N\, J_M$.  One can think of $J_P$ as the subset of $G$ fixing every alcove in ${\mathcal A}$ 
which is contained in the $M$-{\em anti-dominant} alcove ${\bf a}^*_M$ and which is {\em sufficiently antidominant} with respect to the roots in $\Phi_N$. 

Now consider any hyperplane $H_{\beta - k}$ (where we assume $\beta \in \Phi^+$), and let $s_H$ denote the corresponding reflection in $\mathcal A$.  We can define its $J_P$-{\em positive side} $H^+_{\beta-k}$ and its $J_P$-{\em negative side} $H^-_{\beta-k}$, as follows.  If $\beta \in \Phi_M$, we define $H^-_{\beta-k}$ to be the side of $H_{\beta-k}$ containing ${\bf a}^*_M$.  If $\beta \in \Phi_N$, we define $H^+_{\beta-k}$ by
$$
x \in H^+_{\beta-k} \,\, \mbox{iff} \,\, x-s_H(x) \in \mathbb R_{>0} \beta^\vee.
$$

Suppose ${\bf z}'$ and ${\bf z}$ are adjacent alcoves, separated by the wall $H$.  We say the wall-crossing 
${\bf z}', {\bf z}$ is in the $J_P$-{\em positive direction} provided that ${\bf z}' \subset H^-$ and ${\bf z} \subset H^+$.

\begin{defn}
A gallery ${\bf z}_0, {\bf z}_1, \dots, {\bf z}_r$ is in the $J_P$-{\em positive direction} if every wall-crossing ${\bf z}_{i-1}, {\bf z}_i$ is in the $J_P$-positive direction.
\end{defn}

The following lemma is a mild generalization of Lemma 5.3 of \cite{HN}, and its proof is along the same lines.  

\begin{lemma} \label{J_P-pos}
Any gallery ${\bf z}_0, \dots, {\bf z}_r$ in the $J_P$-positive direction is minimal.
\end{lemma}
\begin{proof}
Let $H_i$ denote the wall shared by ${\bf z}_{i-1}$ and ${\bf z}_i$.  If the gallery is not minimal, there exists $i < j$ such that $H_i = H_j$.  We may assume the $H_k \neq H_i$ for every $k$ with $i < k < j$.  By assumption ${\bf z}_{i-1} \subset H_i^-$, and ${\bf z}_i \subset H_i^+$.  Since we do not cross $H_i$ again before $H_j$, we have ${\bf z}_{j-1} \subset H_j^+$.  This is contrary to the assumption that the wall-crossing ${\bf z}_{j-1}, {\bf z}_j$ goes from $H_j^-$ to $H_j^+$.
 \end{proof}

\subsection{Triangles in terms of retractions}

Set $\rho_0 := \rho_{0, \Delta_G}$. Then the space of based triangles 
$\mathcal T(\alpha, \beta; \gamma)$ can be identified with 
$$
\rho_0^{-1}(x_\al) \cap t^\gamma \rho_0^{-1}(x_{\beta^*})
$$ 
since $\rho_0^{-1}(x_\al)= S_\al(o)$ and $t^\gamma \rho_0^{-1}(x_{\beta^*})= S_{\be^*}(x_\ga)$. 

Now we fix a $G$-dominant cocharacter $\mu$ and an $M$-dominant cocharacter $\lambda$ contained in $\Omega(\mu)$.  Fix a parabolic subgroup $P = MN$ having $M$ as Levi factor. Let $\xi\in \tits\B$ be a generic point fixed by $P$. 
Then 
$$\mathcal{IT}(\lambda, \mu; \xi) := \rho_{K_P,\Delta_M}^{-1}(x_\lambda) \cap \rho_0^{-1}(x_\mu)$$
Note that this space depends only on $P$ but not on $\xi$. 
Fix any auxiliary cocharacter $\nu$ which satisfies $\nu \geq^P \mu$. Then Proposition \ref{key_prop} shows that 
we can also describe this space as 
\begin{equation} \label{IT_main_pt}
\mathcal{IT}(\lambda, \mu;\xi) = \rho_{-\nu,\Delta_G-\nu}^{-1}(x_\lambda) \cap \rho_0^{-1}(x_\mu).
\end{equation}

\section{Proof of Theorem \ref{key}} \label{key_thm_proof_sec}

\noindent {\em Proof of Theorem \ref{key}}.
The desired equality
$$
\T(\nu + \lambda, \mu^*;\nu) = t^\nu\big(\I\T(\la, \mu;\xi)\big)
$$
is just the equality
$$
Kx_{\nu+\lambda} \cap t^\nu Kx_\mu = t^\nu \big( K_P x_{\lambda} \cap Kx_\mu \big) 
$$
which can obviously be rewritten as
$$
(t^{-\nu}Kt^\nu)x_{\lambda} \cap Kx_\mu = K_Px_{\lambda} \cap Kx_\mu.
$$
But this follows from Proposition \ref{key_prop}, or more precisely from its equivalent version, the equality stated in Lemma \ref{equiv1}(i).
\qed

\section{On dimensions of varieties of triangles}\label{bounds_sec}

\subsection{Based triangles}

Let $\alpha, \be, \gamma\in \De_G\cap \La$.    Assume that 
$\alpha + \beta - \gamma \in Q(\Phi^\vee)$, which is a necessary condition for 
$\mathcal T(\alpha, \beta; \gamma)$ to be nonempty.  It is clear that
$$
\mathcal T(\alpha, \beta; \gamma) = Kx_\alpha \cap t^\gamma Kx_{\beta^*},
$$
and, as we explained earlier, this shows that $\mathcal T(\alpha,\beta;\gamma)$ has the structure of a finite-
dimensional quasi-projective variety defined over $\mathbb F_p$.

It also shows that we can identify $\mathcal T(\alpha, \beta; \gamma)$ with a fiber of the convolution morphism
$$
{\mathbf m}_{\alpha, \beta}: S_\alpha(o) ~ \widetilde{\times} ~ S_{\beta}(o) \rightarrow \overline{S_{\alpha + \beta}(o)}.
$$
Here 
$$S_\al(o) \widetilde{\times} S_\be(o):=\{(y,z)\in {\mathcal B}^2: y\in S_\al(o), z\in 
S_\be(y)\}.$$
By definition, ${\mathbf m}_{\alpha, \beta}(y,z)=z$. It follows that 
$$
\mathcal T(\alpha, \beta; \gamma) = {\mathbf m}_{\alpha,\beta}^{-1}(x_{\gamma}).
$$

We have the following a priori bound on the dimension of the variety of triangles.

\begin{prop}\label{T-count}
Recall that $\rho$ denotes 
the half-sum of the positive roots $\Phi^+$. Then
$$
{\rm dim} \, \T(\al, \be; \ga)\le \<\rho, \alpha + \beta - \gamma \>.
$$ 
Moreover, the number of irreducible components of $\T(\alpha, \beta; \gamma)$ of dimension 
$\<\rho, \alpha + \beta - \gamma \>$ equals the multiplicity $n_{\al,\be}(\ga)$. 
\end{prop}
\proof This can be proved using the Satake isomorphism and the Lusztig-Kato formula, following the method of \cite{KLM3}, $\S8.4$.  Alternatively, one can invoke the semi-smallness of ${\mathbf m}_{\alpha, \beta}$ and the geometric Satake isomorphism, see \cite{Haines1,Haines2}.  Compare \cite{Haines1}, Theorems 1.1 and Proposition 1.3. \qed

\subsection{Based ideal triangles}

Fix $\mu \in \De_G \cap \Lambda$ and $\la \in \De_M \cap \Lambda$, and assume $\mu - \lambda \in Q(\Phi^\vee)$, which is a necessary condition for $\I\T(\mu,\lambda;\xi)$ to be nonempty.

As for the variety of based triangles $\mathcal T(\alpha, \beta;\gamma)$, we want to give an a priori bound on the dimension of $\I\T(\lambda,\mu;\xi)$ and also give a relation between its irreducible components and the multiplicity $r_\mu(\lambda)$. 

The key input is the following analogue of Proposition \ref{T-count} for the intersections of $N$- and $K$-orbits in ${\rm Gr}^G$.  It is proved by considering (\ref{ct_diag}) and manipulating the Satake transforms and Lusztig-Kato formulas for $G$ and $M$, in a manner similar to \cite{KLM3}, $\S8.4$. 

We set $\F := Nx_\lambda \cap Kx_\mu$, a finite-type, locally-closed subvariety of ${\rm Gr}^G$, defined over $\mathbb F_p$.

\begin{prop} [\cite{GHKR1}] \label{IT-count}
Let $\rho_M$ denote the half-sum of the roots in $\Phi_M^+$.   Then 
$$
{\rm dim} \, \F \leq \langle \rho, \mu + \lambda \rangle - 2\langle \rho_M, \lambda \rangle,
$$
and the multiplicity $r_\mu(\lambda)$ equals the number of irreducible components of $\F$ having dimension equal to 
$\< \rho, \mu + \lambda \> - 2\< \rho_M, \lambda \>$.
\end{prop}

The link between the dimensions of $\F = Nx_\lambda \cap Kx_\mu$ and $\I\T = K_P x_\lambda \cap Kx_\mu$ comes from the following relation between these varieties.   To state it, we first recall that the Iwasawa decomposition
$$
G = N\, M\, K
$$
determines a well-defined {\em set-theoretic} map
\begin{align*}
G/K &\rightarrow M/K_M \\
nmkK &\mapsto mK_M
\end{align*}

We warn the reader that this is {\em not} a morphism of ind-schemes ${\rm Gr}^G \rightarrow {\rm Gr}^M$; however, when restricted to the inverse image of an connected component of ${\rm Gr}^M$, it does induce such a morphism.  (Our reference for these facts is \cite{BD}, especially sections 5.3.28---5.3.30.)  Since any $K_P$-orbit belongs to such an inverse image, the following lemma makes sense.

\begin{lemma} \label{fibration}
The map $nmk \mapsto mK_M$ induces a surjective, Zariski locally-trivial fibration
$$
K_Px_\lambda \cap Kx_\mu \rightarrow K_Mx_\lambda
$$
whose fibers are all isomorphic to $Nx_\lambda \cap Kx_\mu$. 
\end{lemma}

\proof 
Let $K_M^\lambda = (t^\lambda K_M t^{-\la}) \cap K_M$ denote the stabilizer of $x_\la$ in $K_M$.  
The essential point is that the morphism $f:K_M\to K_M/K_M^\la$ is 
a locally trivial principal fibration according to \cite[Lemma 2.1]{Haines2}.
Next note that since $f$ is $K_M$-equivariant, it suffices to trivialize $f$ over a neighborhood $V$ of $x_{\lambda}$. 
According to the result of loc.~cit.~there exists a Zariski open neighborhood $V \subset K_M/K_M^\la$ of $x_{\la}$ and a section $s:V \to K_M$ of $f|f^{-1}(V)$.  Hence for $x' \in V$  we have
$$ s(x') x_{\la} = x'.$$ 
We will now prove that the induced map $f:f^{-1}(V) \to V$ is equivariantly equivalent to the product bundle
$V \times \F \to V$.  Indeed, define $\Phi:f^{-1}(V) \to V \times \F$ by
$$\Phi(x) = (f(x), s(f(x))^{-1}x).$$
Put $f(x): = x'$. Note that $\Phi$ does indeed map to $V \times \F$ because (from the equivariance of $f$) we have
$$f(s(f(x))^{-1}x) = s(f(x))^{-1} f(x) = s(x')^{-1} x' = x_{\la}.$$
Now define $\Psi: V \times \F \to f^{-1}(V)$ by
$$\Psi(x',u) = s(x') u.$$
Clearly $\Phi$ and $\Psi$ are algebraic because $s$ and $f$ are.  The reader will verify that
$\Phi$ and $\Psi$ are mutually inverse.
\qed


Since $K_Mx_\lambda$ is irreducible of dimension $2\langle \rho_M, \lambda \rangle$, we immediately deduce the following Proposition from Proposition \ref{IT-count} and Lemma \ref{fibration}.  It is an analogue of Proposition \ref{T-count} for the ideal triangles.

\begin{prop}\label{bound 2}  We have the inequality
$$
{\rm dim} \, \I\T(\lambda,\mu;\xi) \leq \< \rho, \mu + \lambda \> ,
$$
and the number of irreducible components of $\I\T(\la, \mu;\xi)$ having dimension equal to 
$\langle \rho, \mu + \lambda \rangle$ is the multiplicity $r_\mu(\lambda)$.
\end{prop}


\begin{Remark} 
The structure group of the fiber bundle $f:\I \to K_M/K_M^{\la}$ is $K_M^{\la}$ in the following sense.  
Choose trivializations $\Psi_1: V \times \F \to f^{-1}(V)$ and $\Psi_2: V \times \F \to f^{-1}(V)$ determined
by sections $s_1$ and $s_2$ as above. Since $s_1$ and $s_2$ are $K_M$-valued functions on $V$ we may define a new
$K_M$-valued function $k(x')$ on $V$ by $k(x') = s_1(x')^{-1} s_2(x')$. Hence, there exists a morphism $k:V \to K_M^{\la}$ with
$$ s_2(x') = s_1(x') k(x').$$
It is then immediate that
$$\Psi_1^{-1}\circ \Psi_2(x',u) = (x', k(x')u).$$
\end{Remark}

\section{Geometric interpretations of $m_{\alpha,\beta}(\gamma)$ and $c_\mu(\la)$}  \label{geom_int_sec}

The above section gave the geometric interpretations of the numbers $n_{\al,\be}(\ga)$ and $r_\mu(\la)$, by describing them in terms of certain irreducible components of the varieties $\T(\al,\be;\ga)$ and $\I\I(\la,\mu;\xi)$.  The purpose of this section is to give similar geometric interpretations for $m_{\alpha,\be}(\ga)$ and $c_\mu(\la)$.  This will be used to deduce Theorem \ref{main} from Theorem \ref{key}.

\subsection{Hecke algebra structure constants}

By applying Theorem 8.1 and Lemma 8.5 from \cite{KLM3}, and taking into account that
$|S_{\la}(o)(\mathbb F_q)|=|S_{\la^*}(o)(\mathbb F_q)|$,
we see that
\begin{lemma} \label{str_const} 
$m_{\al,\be}(\ga)=|\T(\al,\be;\ga)(\mathbb F_q)|.$
\end{lemma}

\subsection{The constant term} \label{const_term_sec}

The goal of this subsection is to give a geometric interpretation of the constants $c_\mu(\la)$. 
Fix the weights  $\la\in \De_M$, $\mu\in\De_G$.  In what follows we temporarily abuse notation and write $G, K, M$ etc., in place of $G_q, K_q, M_q$, etc.  The following lemma comes from integration in stages according to the Iwasawa decomposition $G = KNM$.

\begin{lemma} \label{counting_via_const_term}  
We have 
$$
|N x_\lambda \cap Kx_\mu| = q^{\langle \rho_N,\lambda \rangle} c^G_M(f_\mu)(t^\lambda)= 
q^{\langle \rho_N,\lambda \rangle} c_{\mu}(\la). 
$$
\end{lemma}

\begin{proof}  The Iwasawa decomposition $G = KNM$ gives rise to an integration formula, relating
integration over $G$ to an iterated integral over the subgroups $K$, $N$, and $M$, where if
$\Gamma$ is any of these unimodular groups, we equip $\Gamma$ with the Haar measure which gives
$\Gamma \cap K$ volume $1$.   For a subset $S \subset G$, write $1_S$ for the characteristic function of $S$.
Using the substitution $y = knm$ in forming the iterated integral, the left hand side above can be written as
\begin{align*}
\int_G 1_{NK}(t^{-\lambda}y) 1_{K t^{\mu}K}(y) dy &= \int_G 1_{NK}(t^{-\lambda}y^{-1}) 1_{Kt^{\mu}K}(y^{-1})\, dy \\
&= \int_M \int_N \int_K 1_{NK}(t^{-\lambda}m^{-1}n^{-1}k^{-1}) 1_{K t^\mu K}(m^{-1}n^{-1}k^{-1}) \, dk \,dn \, dm \\
&=\int_M \int_N 1_{NK}(m^{-1}n^{-1}) 1_{K t^\mu K}(t^\lambda m^{-1}n^{-1}) \, dn \, dm \\
&= \int_M \int_N 1_{NK}(m) 1_{K t^\mu K}(t^\lambda mn) \, dn \, dm \\
&=\int_N 1_{K t^\mu K}(t^\lambda n) \, dn \\
&= \delta_P^{-1/2}(t^\lambda) c^G_M(f_\mu)(t^\lambda),
\end{align*}
which implies the lemma since $\delta_P^{1/2}(t^\lambda) = q^{-\langle \rho_N, \lambda \rangle}$.
\end{proof}

Now we return to the previous notational conventions, where we distinguish between $G, K, M$ etc. and $G_q, K_q, M_q$, etc.

Set
\begin{align*}
{\mathcal I} &:=\I\T(\la,\mu; \xi), \\
\tau_{\la\mu}  &= |\I\T(\la,\mu;\xi)(\mathbb F_q)|, \\
i_{\la,q} &:= |K_{M,q}: K_{M,q}^\la|,
\end{align*}
where $K_M^{\la}$ is the stabilizer of $x_{\la}$ in $K_M$ and $K^\la_{M,q} := K_{M,q} \cap K^\la_M$. In other words, $i_{\la,q}$ is the cardinality of the orbit $K_{M,q}(x_\la)$, and in particular, 
it is finite.   Recall also that
$$
{\mathcal F}={\mathcal F}_{\la\mu}:=N x_{\la} \cap Kx_\mu. 
$$
By Lemma \ref{fibration} the variety ${\mathcal I}$ fibers over $K_M/K_M^{\la}=K_M(x_\la)$ with fibers isomorphic to ${\mathcal F}$ via the map $f$ which is the restriction of the $N$-orbit map on $NK_M x_{\la}$ to $\I$. In particular,
$$
|\I(\mathbb F_q)|=i_{\la,q}|\F(\mathbb F_q)|. 
$$
It is proved in Lemma \ref{counting_via_const_term} that
$$
|{\mathcal F}(\mathbb F_q)|=q^{\< \rho_N,\la \>} c^G_M(f_\mu)(t^{\la})=  
q^{\< \rho_N,\la \>} c_\mu(\la). 
$$
Set
$$
\varphi(\la,q):=  \frac{1}{q^{\< \rho_N,\la \>} i_{\la,q}}.
$$
By combining Lemma \ref{counting_via_const_term} with Lemma \ref{fibration},
we obtain the following geometric interpretation of $c_{\mu}(\la)$: 

\begin{cor}\label{c:count with P}
$$
c_{\mu}(\la)=q^{-\< \rho_N,\la \>} |{\mathcal F}(\mathbb F_q)|  =
\varphi(\la,q) \tau_{\la,\mu}= \varphi(\la,q) |\I\T(\la,\mu;\xi)(\mathbb F_q)|.
$$
\end{cor}

\section{Proofs of Theorem \ref{main} and Corollary \ref{corollary:main}} \label{Thm_main_proof}

\noindent {\em Proof of Theorem \ref{main}}. 

(i) Since $\nu \geq^P \mu$, by Theorem \ref{key}, the $\mathbb F_p$-varieties 
$\T(\al, \be; \ga)=\T(\nu+\lambda, \mu^*;\nu)$ and  
$\I\T(\lambda, \mu;\xi)$ are isomorphic. Since the cardinalities of their sets of $\mathbb F_q$-rational points are 
$m_{\al,\be}(\ga)$ and  
$$
c_\mu(\lambda) q^{\<\rho_N,\la\>} |K_{M,q}\cdot x_\la|$$
respectively (see Lemma \ref{str_const} and Corollary \ref{c:count with P}), we conclude that
$$
m_{\al,\be}(\ga)=c_\mu(\lambda) q^{\<\rho_N,\la\>} |K_{M,q}\cdot x_\la|.
$$

\medskip 

(ii) We now prove the equality $r_\mu(\lambda) = n_{\nu, \mu}(\nu+\la)$.  First, it is well-known that $n_{\nu,\mu}(\nu + \la) = n_{\nu + \la, \mu^*}(\nu)$.  Next, 
by Proposition \ref{T-count}, $n_{\la+\nu,\mu^*}(\nu)$ is the number of 
irreducible components of $\T(\la+\nu,\mu^*;\nu)$ of dimension 
$$
\<\rho, \la+\nu+\mu^*-\nu\>= \<\rho, \la+\mu\>.$$
On the other hand, by Proposition \ref{bound 2}, $r_{\mu}(\la)$ is the number of irreducible components of $\I\T(\la,\mu;\xi)$ of dimension 
$\<\rho, \la+\mu\>$. 
Since  $\T(\la+\nu,\mu^*;\nu)\cong \I\T(\la,\mu;\xi)$, the equality follows.

\medskip 
(iii) Consider the implication
$$r_\mu(\la)\ne 0 \RA c_\mu(\la)\ne 0.$$

Set $\alpha = \nu + \la$, $\be = \mu^*$, and $\ga = \nu$.  Using parts (i) and (ii), the implication is equivalent to the implication
$$
n_{\al,\be}(\ga) \neq 0 \RA m_{\al, \be}(\ga) \neq 0.
$$
Now if $n_{\al,\be}(\ga) \neq 0$, then by Proposition \ref{T-count} the variety $\T(\al,\be;\ga)$ is nonempty and so $|\T(\al,\be;\ga)(\mathbb F_q)| \neq 0$ for all $q >> 0$.  But the Hecke path model for $\T(\al,\be;\ga)$ then shows that $|\T(\al,\be;\ga)(\mathbb F_q)| \neq 0$ {\em for all $q$} (cf. \cite{KLM3}, Theorem 8.18).  Thus $m_{\la,\be}(\ga) \neq 0$.


\medskip
Consider next the implication 
$$
c_\mu(\la)\ne 0 \RA r_{k \mu}(k \la)\ne 0$$
for $k:=k_\Phi$. If $c_\mu(\la)\ne 0$ then $m_{\la+\nu, \mu^*}(\nu)\ne 0$ for $\nu \ge^P \mu$ as above. By \cite{KM2}, 
$$
m_{\la+\nu, \mu^*}(\nu)\ne 0 \RA n_{k\cdot(\la+\nu), k\cdot \mu^*}(k\cdot \nu)\ne 0. 
$$
Since the inequality $\ge^P$ is homogeneous with respect to multiplication by positive integers, we get
$$
k\nu\ge^P k\mu. 
$$
Therefore, 
$$
n_{k\cdot(\la+\nu), k\cdot \mu^*}(k\cdot \nu)= r_{k\mu}(k\la). 
$$
This concludes the proof of Theorem \ref{main}. \qed 

\medskip

\noindent {\em Proof of Corollary \ref{corollary:main}.} 

Assume that $\mu-\nu\in Q(\Phi_G^\vee)$. 

\smallskip

i. (Semigroup property for $r$).  Suppose $(\lambda_i,\mu_i) \in (\De_M \cap \Lambda) \times (\De_G \cap \Lambda)$ are such that $r_{\mu_i}(\lambda_i) \neq 0$ for $i=1,2$.  For each $i$, choose $\nu_i$ with $\nu_i \geq^P \mu_i$.  If we set $\alpha_i := \nu_i + \lambda_i$, $\beta_i := \mu^*_i$, and $\ga_i = \nu_i$, for $i=1,2$, then Theorem \ref{main}(ii) gives us
$$
n_{\al_i,\be_i}(\ga_i) = r_{\mu_i}(\la_i) \neq 0.
$$ 
It is well-known that the triples $(\alpha,\be,\ga)$ with $n_{\al, \be}(\ga) \neq 0$ form a semigroup, so we have $n_{\al_1+\al_2,\be_1+\be_2}(\ga_1+\ga_2) \neq 0$.  By the semigroup property of $\geq^P$, we have $\nu_1 + \nu_2 \geq^P \mu_1 + \mu_2$, so Theorem \ref{main}(ii) applies again and implies
$$
n_{\al_1+\al_2,\be_1+\be_2}(\ga_1+\ga_2) = r_{\mu_1+\mu_2}(\la_1 + \la_2).
$$
The result is now clear.

\medskip

ii. (Uniform Saturation for $c$). Consider the implication 
$$
c_{N\mu}(N\la)\ne 0 \hbox{~~for some~~} N\ne 0 \quad \RA c_{k_{\Phi}\mu}(k_\Phi \la)\ne 0.$$  
As above, take $\nu\ge^P\mu$  and set $\al:=\nu+\la, \, \be:=\mu^*, \, \ga:=\nu$. Then (with some positive factors $Const_1, Const_2$) 
$$
c_{\mu}(\la)=Const_1 \cdot m_{\al, \be}(\ga), \quad c_{N\mu}(N\la)=Const_2\cdot m_{N\al, N\be}(N\ga).
$$  
Note that our assumption $\mu - \lambda \in Q(\Phi^\vee)$ is equivalent to $\la + \mu^* \in Q(\Phi^\vee)$ and thus to $\al + \be - \ga \in Q(\Phi^\vee)$.  Now the implication follows from the uniform saturation for the structure 
constants $m$ for the Hecke ring ${\mathcal H}_G$ proved in \cite{KLM3}. 


\medskip 

iii. (Uniform Saturation for $r$).  Consider the implication 
$$
r_{N\mu}(N\la)\ne 0 \hbox{~~for some~~} N\ne 0 \quad \RA r_{k^2\mu}(k^2 \la)\ne 0$$  
for $k:=k_\Phi$. Similarly to (ii), this implication follows from the implication
$$
n_{N\al,N\be}(N\ga)\ne 0 \hbox{~~for some~~} N\ne 0 \quad \RA n_{k^2\al, k^2\be}(k^2 \ga)\ne 0$$  
proved in \cite{KM2}. Since for type $A$ root systems $\Phi$, $k_\Phi=1$, it follows that the semigroup 
$\{r_{\mu}(\lambda) \neq 0\}$, is saturated. \qed

\section{Remarks on equidimensionality} \label{equidim_sec}

By \cite{Haines2}, we know that when $\alpha$ and $\beta$ are sums of minuscule cocharacters, then the variety $\mathcal T(\alpha,\beta;\gamma)$ is either empty, or it is equidimensional of dimension $\langle \rho, \alpha+\beta-\gamma \rangle$.  Here we present some analogous results for $\I\T(\la,\mu;\xi)$.

\begin{cor} \label{GL_equidim}
Let $G = {\rm GL}_n$.  Then $\I\T(\la,\mu;\xi)= NK_Mx_\la \cap Kx_\mu$ (resp. $Nx_\lambda \cap Kx_\mu$) is either empty or it is equidimensional of dimension $\langle \rho, \mu+\lambda \rangle$ (resp. $\langle \rho, \mu + \la \rangle -2\langle \rho_M, \lambda \rangle$).
\end{cor}

\begin{proof}
 All coweights for ${\rm GL}_n$ are sums of minuscules.  So for any choice of $\nu$ with $\nu \geq^P \mu$, the cocharacters $\alpha = \nu + \la$ and $\beta = \mu^*$ are sums of $G$-dominant minuscule cocharacters.  Hence the result for $\I\T(\la, \mu;\xi)$ follows using Theorem \ref{key} and the equidimensionality of $\T(\nu+\lambda,\mu^*;\nu)$ proved in \cite{Haines2}.  The statement on $Nx_\la \cap Kx_\mu$ follows from the statement on $\I\T(\la, \mu;\xi)$ using Lemma \ref{fibration}.
\end{proof}

With more work, one can show the following stronger result.  We omit the proof.

\begin{prop} 
Let $G$ be arbitrary and suppose $\mu$ is a sum of minuscule $G$-dominant cocharacters.  Then the conclusion of Corollary \ref{GL_equidim} holds for the pair $(\mu,\lambda)$.  Consequently, we have $r_\mu(\lambda) \neq 0 \Leftrightarrow c_\mu(\lambda) \neq 0$.
\end{prop}
Unlike Corollary \ref{GL_equidim}, this is not a direct application of \cite{Haines2}.  It would be in the situation where we can find $\nu$ such that $\nu \geq^P \mu$ and such that $\nu + \la$ is a sum of minuscules.  However, there is no guarantee we can find $\nu$ with such properties in general.

\section{Appendix}\label{appendix}

To simplify the notation we set $\rho_\xi:=b_{\xi,\De_M}$ and $\rho_{-\nu}:=\rho_{-\nu,\De_G -\nu}$. 
Our goal is to give a geometric proof of Theorem \ref{key}, which can be restated as follows. 

\begin{theorem}\label{equal}
If $\nu \ge^{P} \mu$ then for every geodesic $\ga=\ol{o z}\subset \B_G$ of $\De$-length $\mu$, we have
$$
\rho_{-\nu}|\ga= \rho_{\xi}|\ga. 
$$
\end{theorem}
\proof  Throughout the proof we will be using the concept of a {\em generic} geodesic in a building introduced in \cite{KM2}. 
A geodesic (finite or infinite) $\ga$ in $\B_G$ is {\em generic} if it is disjoint from the codimension 2 skeleton of 
the polysimplicial complex $\B_G$, except for, possibly, the end-points of $\ga$. It is easy to see that generic segments are dense: 
Every geodesic contained in the apartment $\A$  is the limit of generic geodesics in $\A$.  

We next review basic properties of the retractions $\rho_\xi$ and $\rho_\nu$. 
Both maps are isometric when restricted to each alcove in $\B_G$; hence, both maps are 
1-Lipschitz, in particular, continuous. 

Observe that for every  $x\in \B_G$, there exists an apartment $\A_x\subset \B_G$ 
so that $\xi\in \tits \A_x$, $\A_x\cap \A$ has nonempty interior and contains an infinite subray in $\ol{o\xi}$. 
This follows by applying Lemma \ref{d-tilde} to 
a generic geodesic $\ga\subset \A$ asymptotic to $\xi$ and passing through an alcove $\a\subset \A$ containing $o$. 
For such an apartment $\A_x$, there exists a unique isomorphism $\phi_x: \A_x \to \A$ fixing $\A_x\cap \A$. Then $y=\phi_x(x)\in \A$ is independent of the choice of $\A_x$ (although, $\phi_x$ does). 
By the definition of $\be_{\xi,\De_M}$, we see that $\be_{\xi,\De_M}(x)=\be_{\xi,\De_M}(y)$. 
Hence, $\rho_x$ factors as the composition $\rho_{\De_M}\circ \rho_{\xi,\A}$, where $\rho_{\xi,\A}(x)=y$. 
(The map $\rho_{\xi,\A}$ equals the map $\rho_{I_P,\A}$ defined in section 6.3.)

We next make observations about the geometric meaning of the partial order $\ge^{P}$. In what follows it will be convenient to extend the partial 
order $\nu\ge^P \mu$ to arbitrary vectors $\mu$ in $\A$ (not only cocharacters) (we will be still assuming however that $\nu\in X_*(\ul{T})$).   
We will also extend the definition of $x_\la$ from $\la\in \La$ to general vectors $\la$ in the affine space $\A$: 
Given a vector $\la$ in the apartment $\A$, we let $x_\la\in A$ denote the point so that $\ov{ox_\la}=\la$.

Given $\nu\in \La$ we let ${\bf a}_\nu$ 
denote the alcove of $(\A, \widetilde{W})$ with the vertex $x_{\nu}$  
and contained in the negative chamber $-\De_G +\nu$.

\begin{lemma}\label{basic}
Suppose that $\nu$ is annihilated by all roots of $\Phi_M$. 
Then the following are equivalent:

\noindent 1. $$
\nu \ge^{P} \mu
$$

\noindent 2. $C_\mu:= {\rm Conv}(W\cdot x_\mu)\cap \De_M$ is contained in $\De - \nu$.

\noindent 3. For every positive root $\al\in \Phi\setminus \Phi_M$, 
$$
\al |_{C_\mu} \ge \al(x_{-\nu}). 
$$ 

\noindent 4. If a wall $H$ of $(\A, \widetilde{W})$ intersects the set $C_\mu$, then it does not separate $x_{-\nu}$ from $\xi$ 
in the sense that the ray $\ol{x_{-\nu} \xi}$ does not cross $H$.  

\noindent 5. If a wall $H$ of $(\A, \widetilde{W})$ has nonempty intersection with $C_\mu$, then it does not separate $\a_{-\nu}$ from $\xi$ 
in the sense that it does not separate any point of $\a_{-\nu}$ from $\xi$. 
\end{lemma}
\proof The proof is straightforward and is left to the reader. We observe only that for every positive root $\al$, 
$$
\max(\al|\a_{-\nu})=\al(x_{-\nu})=-\<\al, \nu\>. 
$$
Thus, if $\nu\ge^P \mu$, then any wall $H$ of $(\A, \widetilde{W})$ intersecting $C_\mu$ does not separate $\a_{-\nu}$ from $\xi$.  
\qed 

\medskip
The next lemma establishes equality of the retractions $\rho_\xi, \rho_{-\nu}$ on certain subsets of $\B_M$. 

\begin{lemma}\label{L1}

\noindent 1. If $\nu\ge^P \mu$ then the retractions $\rho_\xi, \rho_{-\nu}$ agree on 
${\rm Conv}(W x_\mu)\subset \A$. 

\noindent 2. Suppose that $x\in \B_M$ is such that $\rho_\xi(x)\in \De_G -\nu$. Then, again $\rho_\xi(x)= \rho_{-\nu}(x)$. 
\end{lemma}
\proof 1. Let $x\in {\rm Conv}(W x_\mu)$. There exists $w\in W_M$ such that $x':=w(x)\in \De_M$; then
$$
x'\in C_\mu= {\rm Conv}(W x_\mu) \cap  \De_M. 
$$
Clearly, $x'=\rho_\xi(x)$. On the other hand, since $C_\mu\subset \De_{G}-\nu$, it follows that $\rho_{-\nu}(x)=\rho_{-\nu}(x')= x'$. 

2. The proof is similar to (1). 
First, find $k\in K_M$ such that $k(x)\in \A$ and $w\in W_M$ such that $x'=wk(x)\in \De_M$. Then, by the definition of $\rho_\xi$, $wk(x)=\rho_\xi(x)$.  
On the other hand, $d_{\De}(x_{-\nu},x)=d_{\De}(x_{-\nu}, x')$. Then $\rho_{-\nu}(x)= w'(x')$, for some 
$w'\in W_{-\nu}$. However, by our assumption, $x'\in \De_{G}-\nu$, hence $w'(x') = x'$. 
\qed

\medskip 
Note that, given $\mu\in \De_G$, the segment $\ol{o x_\mu}$ is the limit of generic 
segments $\ol{o x_{\mu_i}}\subset {\rm Conv}(W\cdot x_\mu)$. In particular,  
${\rm Conv}(W\cdot x_{\mu_i}) \subset {\rm Conv}(W\cdot x_\mu)$ and, therefore,
$$
\nu \ge^P \mu \Rightarrow \nu\ge^P \mu_i, \forall i.
$$
Thus, since both the retraction $\rho_\xi, \rho_{-\nu}$ are continuous, it suffices to prove Theorem \ref{equal} for $\mu$ such that the segment $\ga$ is generic. 

We will assume from now on that $\mu$ is generic and $\nu\ge^P \mu$. Moreover, we will assume that $\mu$ is {\em rational}, i.e., 
$\mu\in \La\otimes \QQ$. 


For convenience of the reader we recall the definition of a {\em Hecke path} in the sense of \cite{KM2}. Let $\pi=\pi(t), t\in [0, r]$ be a piecewise-linear path in $\A$ parameterized by its arc-length. At each break-point $t$, the path $\pi$ has two derivatives $\pi'_-(t), \pi'_+(t)$, which are unit vectors in $\A$. 
Then $\pi$ is a {\em Hecke path} if the following holds for each break-point (\cite{KM2}, Definitions 3.1 and 3.26):

1. $\pi'_+(t)=dw(\pi'_-(t))$ for some $w\in \widetilde{W}_{\pi(t)}$, the stabilizer of $\pi(t)$ in $\widetilde{W}$. 

2. Moreover, $w$ is a composition of affine reflections 
$$
w=\si_{m} \circ ... \circ \si_{1}, \quad \si_{i}\in  \widetilde{W}_{\pi(t)},  
$$
each $\si_{i}$ is the reflection in an affine hyperplane $\{\al_i(x)=t_i\}$ through the point $\pi(t)$,  
$\al_i\in \Phi^+$, so that for each $i=1,...,m$
$$
\<\al_i, \eta_i\> <0, \quad i=0,...,m-1,
$$
where $\eta_0=\pi'_-(t)$, $\eta_i:=d\tau_{i}(\eta_{i-1}), i=1,...,m$ and $\eta_m=\pi'_+(t)$. 
Thus, for $\pi'_-(t)\in \De_G$, we have $m=0$ and, hence, $\pi'_-(t)=\pi'_+(t)$; this means that the corresponding Hecke path $\pi$ is 
geodesic as it does not have break-points.

\medskip
 We will need another property of Hecke paths: Suppose that $\pi$ is a rational Hecke path, i.e., it starts at $o$ and ends at a {\em rational point}, i.e, a point in $\La\otimes \QQ$. Then there exists $N\in \NN$ so that the path $N\cdot \pi$ is an LS path in the sense of Littelmann \cite{Li}. The proof consists in unraveling the definition of an LS path as it was done in \cite{KM2} and observing that all break-points of a rational LS path occur at rational points. We will also need the fact that for every geodesic segment $\si\subset \B_G$ and any $\nu\in \La$, the image of $\si$ under the retraction $\rho_{x_{\nu}, \De+\nu}$ is a Hecke path (see \cite{KM2}).   
The next proposition generalizes Lemma \ref{image}(b) from $\mu \in \De_G \cap \Lambda$ to $\mu \in \De_G$.

\begin{prop}\label{key1}
$\rho_\xi(\ga)\subset C_\mu$. In particular, $\rho_{\xi,\A}(\ga)\subset {\rm Conv}( W x_\mu)$.
\end{prop}
\proof We first prove an auxiliary lemma which is a weak version of Theorem \ref{equal}.   (Comp. Lemma \ref{comparison1}.)  

\begin{lemma}\label{L2}
Given $\ga$, if $\nu$ is $M$-central and very $N$-dominant (more precisely, $\<\al, \nu\>\ge const(\ga)$ for all roots 
$\al\in \Phi_N$), 
then  $\rho_\xi |_\ga = \rho_{-\nu} |_\ga$. 
\end{lemma}
\proof Consider geodesic rays $\ol{x\xi}$ from the points $x\in \B_G$ asymptotic to $\xi$. For every such ray there exists a unique point $x'=f_\xi(x)\in \ol{x\xi}$ so that the subray $\ol{x' \xi}$ is the maximal subray in $\ol{x\xi}$ contained in $\B_M$. Explicitly, the map 
$f_\xi$ can be described as follows. First, recall that for $x\in \B_G$, there is a unique point $\bar{x}\in \B_M$ so that $\bar{x}=n(x)$ for some $n\in N$ (even though, the element $n\in N$ is non-unique). Moreover, for every alcove $\a\subset \B_G$, 
the element $n\in N$ can be chosen the same for all $x\in \a$. The function $x\mapsto \bar{x}$ is isometric on each alcove 
in $\B_G$, hence, it is continuous. Now, given $x\in \ga$, find an element $n\in N$ so that $\bar{x}:=n(x)\in \B_M$. 
Then, by convexity of $\B_M$ in $\B_G$, $n(\ol{x \xi})\subset \B_M$.  By the above observation, 
the image $n(\ol{x \xi})$ is independent of the choice of $n$. By convexity of $Fix(n)$, the intersection $Fix(n)\cap n(\ol{x \xi})$ 
is an infinite ray.  

\begin{claim}
For every $n\in N$ such that $\bar{x}=n(x)\in \B_M$, we have 
$$
\ol{x' \xi}= Fix(n)\cap n(\ol{x \xi}), 
$$
where $x'=f_\xi(x)$. 
\end{claim}
\proof Since $n(x')\in n(\ol{x \xi})\subset \B_M$, we have $n(x')=x'$. 
For the same reason, $n$ fixes the entire sub-ray $\ol{x'\xi}$ pointwise. Thus, $\ol{x' \xi}\subset Fix(n)\cap n(\ol{x \xi})$. 
Let $y\in \ol{n(x) x'} \setminus \{x'\}$. Then $n^{-1}(y)\in \ol{x x'}\setminus \{x'\}$ and, hence, does
not belong to the subbuilding $B_M$. Therefore, $y\notin Fix(n)$ and $Fix(n)\cap n(\ol{x \xi}) \subset \ol{x' \xi}$. 
 \qed 

\medskip 
We next claim that the function $f_\xi$ is continuous. Indeed, it suffices to verify its continuity on each alcove $\a\subset \B_G$. 
As observed above, $n$ can be (and will be) taken the same for all points of $\a$. 
Then (by using the action of $K_M$) continuity of $f$ reduces to the following 

\begin{claim}
Let $n\in N$. Then the function $p\mapsto q, p\in \A$ defined by 
$$
\ol{q \xi}= Fix(n)\cap \ol{p \xi}  
$$
is continuous. 
\end{claim}
\proof The statement follows easily from the fact that the fixed-point set of $n$ intersected with $\A$ is a convex polyhedron. \qed 

\medskip
We now apply the continuous function $f_\xi$ to the compact $\ga$. Its image is a compact subset $C'$ of $\B_M$. Thus 
$C'':=\rho_\xi(C')\subset \De_M$ is also compact. Then, for all $M$-central $\nu\in \De_G$ which are sufficiently $N$-dominant (depending on the diameter of $C''$), the set $C''$ is contained in the relative interior of $\De_G -\nu$ in $\De_M$. We then claim that for such choice of $\nu$, $\rho_\xi |_\ga = \rho_{-\nu} |_\ga$. 

For every $x\in \ga$, the segment $\ol{x' x''}:=\ol{x'\xi}\cap \rho_\xi^{-1}(\De_G -\nu)\subset \B_M$ has positive length. According to Part 2 of Lemma \ref{L1}, $\rho_\xi |_{\ol{x' x''}}= \rho_{-\nu} |_{\ol{x' x''}}$. Moreover, $\rho_\xi(\ol{x x''})$ is the unique geodesic segment in $\A$ containing the subsegment $\rho_\xi (\ol{x' x''})$ and having the same metric length as $\ol{x x''}$. (This follows from the fact that $\rho_\xi$ restricts to an isometry on the ray $\ol{x\xi}$.) 

We now claim that the projection $\rho_{-\nu}$ also sends $\ol{x x''}$ to a geodesic segment in $\De_G -\nu$. Indeed, the path $\pi:= \rho_{-\nu}(\ol{x'' x})$  is a Hecke path in $\A$. The unit tangent vector $\tau$ to $\pi$ at $\rho_{-\xi}(x'')$ is contained in $\De_G$ since its opposite (pointing to $\xi$) is contained in $-\De_G$. Then the definition of a Hecke path above implies that $\pi$ is geodesic. 

The retraction $\rho_{-\nu}$ preserves metric lengths of curves \cite{KM2}, therefore, $\pi$ is a geodesic of the same length as $\ol{x x''}$. Hence, $\rho_{-\nu}(x)=\rho_\xi(x)$. Lemma follows. \qed

The only corollary of this lemma that we will use is

\begin{cor}\label{hecke}
$\rho_\xi(\ga)$ is a Hecke path  in $(\A, \widetilde{W})$ of the $\De$-length $\mu$  in the sense of \cite{KM2}. 
\end{cor}
\proof By \cite{KM2}, the retractions $\rho_{-\nu}: \B_G \to \De_G-\nu$ send geodesics in $\B_G$ to Hecke paths preserving $\De$-lengths.  Now, the assertion follows from the above lemma. \qed 

\medskip 
We are now ready to prove Proposition \ref{key1}.  
According to Corollary \ref{hecke}, the image $\rho_\xi(\ga)$ is a Hecke path in $\De_M$ with the initial point $o$. Let $\pi$ be a subpath of $\rho_\xi(\ga)$ starting at the origin $o$. Assume for a moment that $\pi=\pi: [0,1]\to \A$ is an LS path in the sense of Littelmann of the $\De$-length $\be$. Then the terminal point $\pi(1)$ of $\pi$ is a weight of a representation $V_\be^{\widehat{G}}$, 
see \cite{Li}. Therefore, $\pi(1)$ is contained in ${\rm Conv}(W x_\be)\subset {\rm Conv}(W x_\mu)$. Since $\pi$ is contained in $\De_M$, it then follows that $\pi\subset C_\mu$.

More generally, suppose that $\pi$ is a subpath of $\rho_\xi(\ga)$ which terminates at a rational point $\pi(1)\in \La\otimes \QQ$ of the 
apartment $A$. Then there exists $N\in \NN$ so that $N\cdot \pi$  is an LS path of the $\De$-length 
$N \be$, where $\be$ is the $\De$-length of $\pi$. Then, by the above argument, 
$$
N\cdot \pi(1)\in {\rm Conv}(W\cdot x_{N \be}) 
$$
and, hence, $\pi(1)\in {\rm Conv}(W\cdot x_\be)\subset {\rm Conv}(W\cdot x_\mu)$. 
The general case follows by density of rational points in $\rho_\xi(\ga)$. Thus, $\rho_\xi(\ga)\subset C_\mu$. 
The second assertion of Proposition \ref{key1} immediately follows from the first. \qed 



\begin{prop}
For every point $x\in \ga$ there exists an apartment $\A_x\subset \B_G$ connecting $\a_{-\nu}, x$ and $\xi$,  i.e., $\a_{-\nu}\cup \{ x\} \subset \A_x$ and $\xi\in \tits \A_x$. 
\end{prop} 
\proof Clearly, the assertion holds for $x=o$ since $o, \a_{-\nu}, \xi$ belong to the common model apartment $\A_o:=\A \subset X$. 
We cover $\ga$ by alcoves $\a_i\subset \B_G$, $i=0,...,m$, where $\a_0\subset \A$ is an alcove intersecting $\ga$ only at the point $o$. 
Recall that, by the genericity assumption, $\ga_i=\a_i\cap \ga$ is contained in the interior of $\a_i$ 
except for the end-points of this arc. Then, if the assertion holds for some point in the interior of $\ga_i$, it holds for all points of $\ga_i$. We suppose therefore that the assertion holds for points in the alcoves $\a_0,...,\a_k$ and will prove it for the points of $\ga_{k+1}$. We will mostly deal with the case $k\ge 1$ and explain how to modify the argument for $k=0$. 
Let $x:=\ga_k\cap \ga_{k+1}$ and $y$ be such that $\ol{xy}=\ga_{k+1}$.  

Let $\A_x$ be an apartment as above. We assume that $x$ belongs to a wall $H$ in $\B_G$ and $\a=\a_{k+1}$ is not contained in $A_x$ 
(for otherwise we again would be done). If $k=0$, we take $H\subset \A$. In this case, $\nu\ge^P\mu$ implies that $H$ does not separate $\xi$ from $\a_{-\nu}$. Assume now that $k>0$. Since $\ga$ is generic, the germ $H\cap \a$ of $H$ at $x$ is contained in $\A_x$. Therefore, without loss of generality, we can assume that $H\subset \A_x$. 

\begin{claim}
$H$ does not separate $\a_{-\nu}$ from $\xi$ in $\A_x$. 
\end{claim}
\proof In view of Lemma \ref{basic}, it suffices to show that $H$ does not separate $x_{-\nu}$ from $\xi$. 

Let $H'\subset \A$ be the (unique) wall containing $\rho_\xi(H\cap \a)$. 
Since $\nu\ge^P \mu$, $\rho_\xi(\ga)\subset C_\mu= {\rm Conv}(W x_\mu)\cap \De_M)$ (Proposition \ref{key1}), it follows that $H'\cap C_\mu\ne \emptyset$. Hence,  $H'$ does not separate $\a_{-\nu}$ from $\xi$. 

We recall that the map $\rho_\xi | \A_x$ is obtained in two steps: First, an isomorphism $\phi: \A_x\to \A$ fixing $\a_{-\nu}$, and then  
applying the projection $\rho_{\De_G-\nu}: \A\to \De_G-\nu$ (obtained by acting on $\phi(p), p\in \A_x$, by an appropriate element $w\in W_M$). Let $w\in W_M$ be the element which sends $\phi(\a_k)$ (and, hence, $\phi(H\cap \a)$) to $\De_M$. Note that $w$ fixes $\xi$ and $x_{-\nu}$. Since $H'$ did not separate $\xi$ from $x_{-\nu}$ in $\A$, it then follows that $w^{-1}(H')$ does not separate either. Since $\phi: \A_x\to \A$ is an isomorphism fixing $\xi$ and $\a_{-\nu}$, it then follows that $H=\phi^{-1}w^{-1}(H')$ also does not separate 
$\xi$ and $x_{-\nu}$.  The claim follows. \qed 

\medskip 
Since $\B_G$ is a thick building, there exists a half-apartment $\A^+_y\subset \B_G$ containing the alcove $\a$, so that $\A^+_y\cap \A_x=H$. Let $\A_x^-$ denote the half-space in $\A_x$ bounded by $H$ and containing $\a_{-\nu}$; hence,  $\xi\in \geo \A_x^-$. 
Then $\A_y:=\A_x^-\cup \A_y^+$ is an apartment, $\a_{-\nu}\subset \A_y, y\in \A_y$ and  $\xi\in \geo \A_y$. Proposition follows. \qed 

\medskip
We now can finish the proof of the main theorem. Pick $x\in \ga$. We will show that $\rho_{-\nu}(x)=\rho_\xi(x)$. 

The map $\rho_{-\nu}: \B_G \mapsto \De_G-\nu$ is the composition of two maps: First 
the canonical isomorphism of the apartments $\psi_x: \A_x\to \A$ which fixes the intersection $\A_x\cap \A$, 
and then the quotient map $\rho_{\De_G-\nu}: \A\to \De_G-\nu$. The intersection 
 $V:=\A_x\cap \A$ has nonempty interior in  $\A$ (since $\a_{-\nu}$ does) 
 Similarly, the projection $\rho_\xi : \B_G \mapsto \De_M$  is obtained by first taking the isomorphism of apartments 
$$
\rho_{\xi,\A}|\A_x: \A_x\to \A$$ 
(again, fixing $V$) and then applying $\rho_{\De_M}$. 
Since $V$ has nonempty interior, it follows that the isomorphisms of apartments 
$\psi_x$ and $\rho_{\xi,\A_x}|\A_x$ agree on the entire apartment $\A_x$. Hence, 
$$
\rho_{\xi,\A}(x)= \psi_x(x). 
$$
By Proposition \ref{key1}, $\rho_{\xi,\A}(x)\in {\rm Conv}(W x_\mu)$. By Lemma \ref{L1}, Part 1, 
$$
\rho_{\De_M}|{\rm Conv}(W x_\mu)= \rho_{\De_G-\nu}|{\rm Conv}(W x_\mu). 
$$
Therefore, 
$$
\rho_\xi(x)=\rho_{\De_M}\circ \rho_{\xi,\A}(x)= \rho_{\De_G-\nu}\circ \psi_x= \rho_{-\nu}(x). \qed 
$$

\small
\bigskip
\obeylines
\noindent
Thomas J. Haines: University of Maryland
Department of Mathematics
College Park, MD 20742-4015, U.S.A.
email: tjh@math.umd.edu

\bigskip
M. Kapovich: Department of Mathematics
University of California
1 Shields Ave,
Davis, CA 95616, USA
email: kapovich@math.ucdavis.edu

\bigskip
John J. Millson: University of Maryland
Department of Mathematics
College Park, MD 20742-4015, U.S.A.
email: jjm@math.umd.edu

\end{document}